\documentclass[11pt,a4paper]{article}
\usepackage[english]{babel}
  \usepackage{amsmath,amsfonts,amssymb,relsize,amsthm}
  \usepackage{esint}

  \usepackage{hyperref} 

 \usepackage[active]{srcltx} 
 \usepackage{color,soul} 

\newtheorem{theorem}{Theorem}[section]

\newtheorem{lemma}[theorem]{Lemma}
\newtheorem{proposition}[theorem]{Proposition}

\theoremstyle{definition}
\newtheorem{definition}[theorem]{Definition}
\newtheorem{remark}{Remark}

\numberwithin{equation}{section}

\def\O{{\Omega}}
\def\o{{\omega}}
\def\eps{{\varepsilon}}

\def\l{{\mathcal{L}}}
\def\m{{\mathcal{M}}}

\def\s{{\mathcal{S}}}
\def\E{{\mathcal{E}}}
\def\L{{\mathcal{L}}}
\def\M{{\mathcal{M}}}

\def\R{{\mathbb{R}}}

\def\N{{\mathbb{N}}}

\newcommand{\ope}[1]{\E[{#1}]}

\newcommand{\lb}[1]{\L_{_{#1}}}
\newcommand{\oplb}[2]{\L_{_{#2}}[{#1}]}
\newcommand{\oplr}[2]{\L_{_{#2}}[{#1}]}

\newcommand{\opm}[1]{\M[{#1}]}

\newcommand{\opmem}[2]{\M_{_{\eps,{#2}}}[{#1}]}
\newcommand{\nlp}[3]{\|{#1}\|_{L^{#2}{(#3)}}}

\newcommand{\nlto}[1]{\|{#1}\|_{2}}

\newcommand{\transposee}[1]{{\vphantom{#1}}^{\mathit t}{#1}}

\newenvironment{formula}[1]{\begin{equation}\label{eq:#1}}
                       {\end{equation}\noindent}

\def\Fi#1{\begin{formula}{#1}}
\def\Ff{\end{formula}\noindent}

\setstcolor{red}

\title{Persistence criteria for populations with  non-local  dispersion 
}

\author{Henri Berestycki \thanks{
CAMS - \'Ecole des Hautes \'Etudes en Sciences Sociales,
190-198 avenue de France, 75013, Paris, France, {\itshape email:
}\ttfamily hb@ehess.fr} , J\'er\^ome Coville \thanks{
UR 546 Biostatistique et Processus  Spatiaux, INRA, Domaine St Paul Site Agroparc, F-84000 Avignon, France, {\itshape email:
}\ttfamily jerome.coville@avignon.fr} , Hoang-Hung Vo \thanks{CAMS - \'Ecole des Hautes \'Etudes en Sciences Sociales,
190-198 avenue de France, 75013, Paris, France, {\itshape email:
}\ttfamily hhv@ann.jussieu.fr}}

\begin{document}

\maketitle
\begin{abstract}
In this article,  we analyse  the non-local   model :
$$
 \frac{\partial u}{\partial t}=J\star u -u + f(x,u) \quad \text{ with }\quad x \in \R^N,
$$
where $J$ is a positive continuous dispersal kernel and $f(x,u)$ is a heterogeneous KPP type non-linearity describing the growth rate of the population.  The ecological niche of the population is assumed to be bounded (i.e. outside a compact set, the environment is assumed to be lethal for the population). For compactly supported dispersal kernels $J$, we derive an optimal persistence criteria.  We prove that  a positive stationary solution exists if and only if the generalised principal eigenvalue $\lambda_p$  of the linear problem
$$ J\star \varphi(x) -\varphi(x) + \partial_sf(x,0)\varphi(x)+\lambda_p\varphi(x)=0 \quad \text{ in }\quad \R^N,$$
 is negative. $\lambda_p$ is a spectral quantity that we defined in the spirit of the generalised first eigenvalue of an elliptic operator. In addition,  for any continuous  non-negative initial data that is bounded or integrable, we establish the long time behaviour of the solution $u(t,x)$. We also analyse the impact of the size of the support  of the dispersal kernel on the persistence criteria. We exhibit situations where the dispersal strategy has "no impact" on the persistence of the species and other ones  where the slowest dispersal strategy is not any more an "Ecological Stable Strategy". We also discuss  persistence criteria for fat-tailed kernels. 
\end{abstract}
\tableofcontents
\section{Introduction}
In this article, we are interested in finding persistence criteria for a species that has a long range dispersal strategy.  For such a model species, we can think of trees of which seeds and pollens are disseminated on a wide range.  
The possibility of a long range dispersal is well known in ecology, where numerous data now available support this assumptions \cite{Cain2000,Clark1998,Clark1998a,Schurr2008}. A commonly used model that integrates such long range dispersal  is the following {\em nonlocal reaction diffusion equation} (\cite{Fife1996,Grinfeld2005,Hutson2003,Lutscher2005,Turchin1998}): 
 \begin{equation}
 \frac{\partial u}{\partial t}(t,x)=J\star u(t,x) -u + f(x,u(t,x)) \quad \text{ in }\quad \R^+\times \R^N. \label{bcv-eq-dyn}
 \end{equation}
Here $u(t,x)$ is  the density of the considered population, $J$ is a dispersal kernel, $f(x,s)$ is a KPP type non-linearity describing 
the growth rate of the population.  

In this setting the tail of the kernel can be thought of as the range of dispersion or as  a measure of  the frequency at which  long dispersal events occur. A biological  motivation  for the use of \eqref{bcv-eq} to describe the evolution of the population comes from the observation  that the  intrinsic variability in the capacity of the individuals to disperse generates,  at the scale of a population, a long range dispersal of the population. The effect of such variability has been investigated in \cite{Hapca2009,Petrovskii2011}  by means of correlated random walks. In such a framework,  each individual moves according to  classical random walks, however  the diffusion coefficients is given by  a probability law. It can be checked (\cite{Hapca2009,Petrovskii2011,Turchin1998}) that  the density of population will then satisfy an integro-differential equation  where the dispersal kernel $J$ describes the probability to jump from one location to another.      

Throughout this paper we will always make the following assumptions on the dispersal kernel $J$.

\textbf{(H1)}  $J \in C(\R^N)\cap L^1(\R^N)$ is nonnegative,   symmetric and of unit mass (i.e. $\int_{\R^N}J(z)dz=1$) . 

\textbf{(H2)} $J(0)>0$ 

In the present paper,  we focus our analysis on species that have a bounded ecological niche. A simple way to model  such a
spatial repartition consists in considering  that the environment is hostile to the species outside a bounded set. For instance, biological populations that are sensitive to temperature thrive only in a limited latitude zone. Thus, if $x$ is the latitude, we get such dependence.   
This fact is translated in our model  by assuming  that $f$ satisfies:

\textbf{(H3)} $f \in C^{1,\alpha}(\R^{N+1})$  is of KPP type, that is :
$$\left\{
\begin{aligned}
& \hbox{$f(\cdot,0)\equiv0$},
\\
& \hbox{For all $x\in \R^N,f(x,s)/s$ is  decreasing  with
respect to $s$ on $(0,+\infty)$}.
\\
& \hbox{There exists $S(x)\in C(\R^N)\cap L^{\infty}(\R^N)$ such that $f(x,S(x))\leq 0$ 
for all $ x\in \R^N$.}
\end{aligned}
\right.$$
 
 
  \textbf{(H4)} \qquad
  \hbox{ $\limsup_{|x|\to \infty} \frac{f(x,s)}{s}<0,$\quad \text{ uniformly in } \quad $s\ge 0.$} 
   
\medskip  
 
A typical example of such a nonlinearity is given by $f(x,s):=s(a(x) - b(x)s)$ with $b(x)>0$ and $a(x)$ satisfies
$\limsup_{|x|\to \infty} a(x)<0$.

 Our main purpose is to find conditions on $J$ and $f$ that characterise the persistence of the species modelled by \eqref{bcv-eq-dyn}. In this task,  we focus our analysis on the description of the set of positive stationary solutions of \eqref{bcv-eq-dyn}, that is the  positive solutions of the equation below 
\begin{equation}
 J\star u(x) -u(x) + f(x,u(x))=0 \quad \text{ in }\quad  \R^N. \label{bcv-eq}
 \end{equation}
Existence of stationary solutions  is naturally expected to provide the right persistence criterion. We will see that this is indeed the case.    

In the literature,  persistence criteria  have been well studied for the local reaction diffusion version of \eqref{bcv-eq-dyn} 
 \begin{equation}\label{bcv-eq-rd}
 \frac{\partial u}{\partial t}(t,x)=\Delta u(t,x) + f(x,t,u(t,x)) \quad \text{ in }\quad \R^+\times \O, 
 \end{equation}
 where $\O$ is a domain of $\R^N$, possibly $\R^N$ itself.
Persistence criteria have been obtained for various media, ranging from periodic media to ergodic media \cite{Berestycki2005,Berestycki2007,Berestycki2008,Cantrell1991,Cantrell1998,Nadin2012,Pinchover1988,Shen2004}. In the context of global warming, persistence criteria have been investigated in  \cite{Berestycki2008,Berestycki2009,Berestycki2009b}.
For such reaction diffusion equations the persistence criteria are often derived from the sign of the first eigenvalue of the linearised problem  at the 0 solution. One is thus led to determine the sign of the first eigenvalue  $\lambda_1(\Delta+ \partial_sf(x,0),\O)$ of the spectral problem 
\begin{equation}\label{bcv-eq-rd-lin0}
\Delta \varphi(x) +\partial_sf(x,0)\phi(x) +\lambda_1 \varphi(x)=0 \quad \text{ in }\quad \O 
 \end{equation}
 associated with the proper boundary conditions (if $\Omega\not=\R^N$).
 
In most situations, for KPP-- like non-linearities, the existence of a positive stationary solution to \eqref{bcv-eq-rd} is indeed uniquely conditioned by the sign of $\lambda_1$. More precisely, there exists  a unique positive stationary solution if and only if  $\lambda_1<0$. If such type of criteria seems reasonable for problems defined on bounded set, it is less obvious for problems in unbounded domains.  In particular, in unbounded domains, one of the main difficulty concerns the definition of $\lambda_1$.  As shown in \cite{Berestycki2006,Berestycki2007,Berestycki2010},  the notion of first eigenvalue in unbounded domain can be quite delicate and several definitions of $\lambda_1$ exist rendering the question of  sharp persistence criteria quite involved.

Much less is known for the non-local equation \eqref{bcv-eq}  and, to our knowledge persistence criteria have been essentially investigated 
in some specific situations such as periodic media : \cite{Coville2008b,Coville2013,Shen2012} or for a version of the problem \eqref{bcv-eq} defined in a bounded domain $\O$ \cite{Bates2007,Coville2010,2013arXiv1305.7122C,Garcia-Melian2009,Kao2010,Shen2012} : 
\begin{equation}\label{bcv-eq-bounded}
 \frac{\partial u}{\partial t}(t,x)=\int_{\O}J(x-y)u(t,y)\,dy -u(t,x) + f(x,u(t,x)) \quad \text{ in }\quad \R^+\times \O. 
 \end{equation}
   
  We also quote \cite{Berestycki2011} for an analysis of a persistence criteria in periodic media for a non-local version of \eqref{bcv-eq-rd} involving a fractional diffusion and \cite{Rawal2012} for persistence criteria  in time periodic versions of \eqref{bcv-eq-bounded} .    
Similarly to the local diffusion case, for $KPP$ like non-linearities, the existence of a positive  solution of the non-local equation \eqref{bcv-eq-bounded} can be characterised  by the sign of a spectral quantity $\lambda_p$, called  the generalised principal eigenvalue of 
\begin{equation}\label{bcv-eq-lin0}
\int_{\O}J(x-y) \phi(y) \, dy - \phi(x) +\partial_sf(x,0)\phi(x) +\lambda \phi(x)=0 \quad \text{in}\quad\O. 
 \end{equation}

In the spirit of \cite{Berestycki1994}, this generalised principal eigenvalue $\lambda_p$ is defined by :
$$\lambda_p:=\sup\left\{\lambda\in\R\, |\, \exists \varphi \in C(\O), \varphi>0,\; \text{ such that }\;\oplb{\varphi}{\O}(x) - \varphi(x) +\partial_sf(x,0)\varphi(x) +\lambda \varphi(x)\le 0 \quad \text{in}\quad \O. \right\},$$ 
where $\oplb{\varphi}{\O}(x)$ denotes  
$$\oplb{\varphi}{\O}(x):= \int_{\O}J(x-y)\varphi(y)\,dy.$$
  
 Unlike the elliptic PDE case, due to the lack of a regularising effect of the diffusion operator,  the above spectral problem may not have a solution in  spaces of functions like $L^p(\O), C(\O)$\cite{Coville2008b,Coville2013a,Kao2010}. As a consequence, even in a bounded domain, simple sharp persistence criteria are  quite delicate. Another difficulty inherent to the study of nonlocal equations  \eqref{bcv-eq-lin0} in unbounded domain concerns the lack  of natural \textit{ a priori } estimates for the solution thus making standard approximations difficult to use in most cases.

\subsection{Main Results:}
Let us now state our main results. In the first one we establish a simple sharp persistence criteria assuming that the dispersal kernel $J$ has compact support.    

\begin{theorem}\label{bcv-thm1}
Assume that $J,f$ satisfy (H1-H4) and assume further that $J$ is compactly supported.  Then,  there exists a  positive solution, $\tilde u$, of \eqref{bcv-eq} if and only if $\lambda_p(\M+\partial_sf(x,0))<0$, where $\M$ denotes the continuous operator $\opm{\varphi}=J\star \varphi(x)-\varphi(x)$ and

$$\lambda_p(\M+\partial_sf(x,0)):=\sup \{\lambda\in \R\,|\, \exists \varphi \in C(\R^N), \phi>0\; \text{ so that }\; \opm{\varphi}+\partial_sf(x,0)\varphi+\lambda\varphi\le 0\}.$$
When it exits, the solution is unique, that is, 
if $v$ is another bounded solution, then $\tilde u=v$ almost everywhere. Moreover, for any non-negative initial data $u_0 \in C(\R^N)\cap L^{\infty}(\R^N)$ we have the following asymptotic behaviour:
\begin{itemize}
\item If $\lambda_p(\M+\partial_sf(x,0))\ge 0$, then  the solution satisfies $\|u(t)\|_{\infty}\to 0$ as $t\to \infty$,
\item If $\lambda_p(\M+\partial_sf(x,0))< 0$, then  the solution satisfies $\|(u- \tilde u)(t)\|_{\infty}\to 0$ as $t\to \infty$.
\end{itemize}
In addition, if the initial data is such that $u_0 \in C(\R^N)\cap L^{1}(\R^N)$, then  the convergence $u(t,x)\to \tilde u$ holds in $L^1(\R^N)$. 
  
\end{theorem}

We observe that the stationary solution  $\tilde u$ may not necessarily be continuous. For some homogeneous problems $f(x,u)=f(u)$, it is known that discontinuous solution may exists \cite{Coville2008a}.  
If $f$ satisfies the stronger hypothesis that, for any $x$, $f(x,u)$ is
concave with respect to $u$, then actually the solution
$\tilde u$ is continuous. To see this, it suffices to notice that $J \star \tilde u >0$ in $\R^N$. The concavity of $f$ with respect
to $u$ implies that for any $x$ the map $ u \mapsto u - f(x,u)$ is
strictly increasing whenever $u - f(x,u)>0$. Then from the
continuity of $J \star \tilde u$ and \eqref{bcv-eq},
which can be rewritten as in the form $J \star \tilde u = \tilde u - f(x,\tilde u)$, we
deduce that $\tilde u$ is continuous.
\bigskip

Next, we aim at understanding the effect of the dispersal kernel on the persistence of the species, more precisely, of its range and scaling.
To this end, we analyse the behaviour of the persistence criteria under some  scaling of the dispersal operator. 
 More precisely, let $J_\eps:= \frac{1}{\eps^N}J\left(\frac{z}{\eps}\right)$ and let $\M_\eps$ denotes the operator $\m$ with the rescaled kernel, that is, $\m_{\eps}[\varphi]:=J_\eps\star \varphi -\varphi$.
  We are interested in the behaviour of the solution to \eqref{bcv-eq}  as $\eps \to 0$ or $\eps \to +\infty$ where the dispersal operator $\m$ is replace by $\gamma(\eps,m)\M_\eps$, with $\gamma(\eps,m)\sim \frac{\alpha_0}{\eps^m}$. 
\smallskip
  
  These asymptotics represent two possible strategies that are observed in nature. The dispersal kernel  $\gamma(\eps,m)J_\eps$  arises when the dispersal of the species is conditioned by  \textit{ a dispersal budget} as defined in \cite{Hutson2003}.  Roughly speaking, for a fixed cost,  this budget is a way to  measure  the differences between different range strategies. For a given cost function of the order of $|y|^m$, the term $\gamma(\eps,m)$ behaves like $ \frac{\alpha_0}{\eps^m}$ and in the analysis, the dispersal operator  is then  given by $\gamma(\eps,m)\M_\eps$. As explained in \cite{Hutson2003},  the limit as $\eps \to 0$ can be associated to a strategy of producing many  offspring but with little capacity of movement. On the other extreme of the spectrum, the limit $\eps \to +\infty$ corresponds to a strategy that aims at maximizing the exploration of  the environment at the expense of the number of offspring produced. We will discuss more precisely this notion of \textit{ dispersal budget} in Section \ref{bcv-section-bio}, where we also interpret our findings in this context.    
  For simplicity, we introduce the notation $\m_{\eps,m}$ to denote the operator $\gamma(\eps,m)\m_\eps$.

In the present paper, we analyse  the cases $0\le m\le 2$ and $\alpha_0=1$.  The study of the case $m=0$  corresponds to understanding the impact of the mean distance by itself on the persistence criteria. To simplify the presentation of these asymptotics, we restrict our discussion to nonlinearities $f(x,s)$ of the form 
$$f(x,s)=s(a(x)-s).$$ 

However, most of the proofs apply to a more general nonlinearity $f(x,s)$, and with $a(x)=\partial_sf(x,0)$. 

Our first result deals with the case $m=0$. 

\begin{theorem}\label{bcv-thm3}
Assume that $J$ and $f$ satisfy (H1-H4), $J$ is compactly supported and let $m=0$. Then there exists $\eps_0\in (0,+\infty]$ so that for all $\eps<\eps_0$ there exists a positive solution $u_\eps$ to \eqref{bcv-eq}. Moreover, at this value $\eps_0$, we have 
$$\lim_{\eps\to \eps_0}u_\eps(x)=(a(x)-1)^+,$$
where $s^+$ denotes the positive part of $s$ (i.e. $s^+=sup\{0,s\}$).
Assuming further that $a$ is smooth, Lipschitz continuous, we have 
$$\lim_{\eps\to 0}u_\eps(x)= v(x) \quad \text{almot everywhere}$$
where $v$ is a non-negative bounded solution of 
$$v(x)(a(x)-v(x))=0 \quad \text{ in }\quad \R^N.$$
When $\eps_0<+\infty$ and, in addition, $a(x)$ is  symmetric ($a(x)=a(-x)$ for all $x$) and the map $t\to a(tx)$ is non increasing for all $x,t>0$,  then $\eps_0$ is sharp, in the sense that for all $\eps \ge \eps_0$
there is no positive solution of \eqref{bcv-eq}.
\end{theorem}

The ecological interpretation of this result bears on the single range of expansion factor. It shows that a strategy for  species to persist  is to match the resource and not to move much.  
Note that it can happen that  $\eps_0=+\infty$ and then there is no effect of the dispersal on the persistence criteria of the species. 
A natural sufficient condition for this to happen  is $$(a(x)-1)^+\neq 0.$$
In this context, the birth rates exceed all death rates and this guarantees the persistence of the population regardless the dispersal strategy. In particular, there exists a bounded positive solution to \eqref{bcv-eq} for any positive kernel $J$. The uniqueness and the behaviour at infinity of the solution are  still open questions for general kernels. 

When $m>0$, the characterisation of the existence of a positive solution changes and a new picture emerges. In particular, for large $\eps$ there is always a positive solution of \eqref{bcv-eq}, whereas for small $\eps$, when $m=2$, it may happen  that no positive solution exists.
Thus, when $m=2$,the situation is, in a sense, opposite to the case when $m=0$. 
Here is our precise results.
\begin{theorem}\label{bcv-thm4}
Assume that $J$ and $f$ satisfy (H1-H4), $J$ is compactly supported and  let $0<m< 2$. There exist $\eps_0\le \eps_1 \in (0,+\infty)$ such that for all $\eps\le \eps_0$ and all $\eps\ge \eps_1$ there exists a positive solution $u_\eps$ of \eqref{bcv-eq}. Moreover, we have 

$$
\lim_{\eps\to +\infty}\|u_\eps-a^+\|_{\infty}=0,\qquad \lim_{\eps\to +\infty}\|u_\eps-a^+\|_{L^2(\R^N)}=0.
$$

In addition,  assuming further that $a$ is  $C^{2}(\R^N)$, we have 
$$\lim_{\eps\to 0}u_\eps(x)= v(x) \quad \text{almost everywhere},$$
where $v$ is a non-negative bounded solution of 
$$v(x)(a(x)-v(x))=0 \quad \text{ in }\quad \R^N.$$

\end{theorem}

In the next Theorem, we require the following notation for the second moment of $J$:
$$D_2(J):=\int_{\R^N}J(z)|z|^2\,dz.$$

\begin{theorem}\label{bcv-thm5}
Assume that $J$ and $f$ satisfy (H1-H4), $J$ is compactly supported and  let $m=2$. Then, there exists $\eps_1\in (0,\infty)$ so that for all $\eps\ge \eps_1$ there exists a positive solution $u_\eps$ to \eqref{bcv-eq}. Moreover, 
$$\lim_{\eps\to +\infty}u_\eps=a^+(x).$$
In addition, if $J$ is radially symmetric, we have the following dichotomy 

\begin{itemize}
\item When $\lambda_{1}\left(\frac{D_2(J)}{2N}\Delta+a(x)\right)<0,$ there exists $\eps_0\in (0,\infty)$ such that for all $\eps\le \eps_0$ there exists a positive solution of \eqref{bcv-eq} and 
$$u_\eps\to v,\quad \text{ in } \quad L_{loc}^2(\R^N), $$
where $v$ is the unique bounded non-trivial  solution of  
$$\frac{D_2(J)}{2N}\Delta v+v(a(x)-v)=0 \quad \text{in}\quad \R^N. $$

\item When $\lambda_{1}\left(\frac{D_2(J)}{2N}\Delta+a(x)\right)>0$
there exists $\eps_0\in (0,\infty)$ such that for all $\eps\le \eps_0$ \eqref{bcv-eq} does not have any positive solution. 
\end{itemize}
\end{theorem}

This last  result clearly highlights the dependence of the spreading strategy on the cost functions and the structure of ecological niche. Especially  when $m=2$,  the smaller spreader strategy may not be an optimal strategy, in the sense that a population adopting such strategy can go  extinct. This effect will be discussed in more detail in the next Section.  
\medskip

Lastly, we further establish  existence/ non-existence criteria when we relax the compactly supported constraint on the dispersal kernel $J$. In this direction, we investigate a class of kernel $J$ that can have a fat tail but still have some  decay at infinity. More precisely, we assume that
\medskip
  
  \textbf{(H5)} \qquad
  \hbox{$ \qquad\int_{\R^N}J(z)|z|^{N+1}<+\infty.$}
  \medskip

\begin{theorem}\label{bcv-thm2}
Assume that $J$ and $f$ satisfy (H1-H4) and assume further that $J$ satisfies $(H5)$.  Then   
\begin{itemize}
\item[(i)]   if $\lambda_p(\M +\partial_sf(x,0))>0$  there is no bounded positive solution of \eqref{bcv-eq},
\item[(ii)] if $\lim_{R\to \infty}\lambda_p(\lb{R}+\partial_sf(x,0))<0$ then there exists a unique positive solution of \eqref{bcv-eq}, where $$\oplb{\varphi}{R}:=\int_{B_R(0)}J(x-y)\varphi(y)\,dy -\varphi(x).$$
\end{itemize}
\end{theorem}

\subsection{Remarks on the Principal eigenvalue}

Before going into the proofs of these results, let us comment on the notion of generalised principal eigenvalue. Our  results essentially hinge on the properties of the principal eigenvalue $\lambda_p(\M+a(x))$ and  more precisely on the relations between the following spectral quantities:
\begin{align*} 
&\lambda_p(\m +a):=\sup\left\{\lambda\in\R\, |\, \exists \varphi \in C(\O), \varphi>0,\; \text{ such that }\;\opm{\varphi}(x) +a(x)\varphi(x) +\lambda \varphi(x)\le 0 \quad \text{in}\quad \O \right\}.
\\
&\lambda_p'(\m +a):=\inf\left\{\lambda\in\R\, |\, \exists \varphi \in C(\O)\cap L^{\infty}(\O), \varphi>0,\; \text{ such that }\;\opm{\varphi}(x) +a(x)\varphi(x) +\lambda \varphi(x)\ge 0 \quad \text{in}\quad \O \right\}.\\
&\lambda_v(\m +a):=\inf_{\varphi \in L^2(\R^N),\varphi \not\equiv 0} \frac{\frac{1}{2}\iint_{\R^N\times\R^N}J(x-y)[\varphi(x)-\varphi(y)]^2\,dxdy -\int_{\R^N}a(x)\varphi^2(x)\,dx}{\nlto{\varphi}^2}.
\end{align*}

These quantities have been introduced in various contexts (see for example \cite{Coville2010,Coville2013,2013arXiv1305.7122C,Garcia-Melian2009a,Ignat2012}). However until now,  relations between them have not been fully investigated or only in some particular contexts such as when  $a(x)$ is homogeneous or periodic. 
Some new results have been recently obtained in \cite{Berestycki2014} now allowing us to have a clear description of the relation between $\lambda_p,\lambda_p'$ and $\lambda_v$. Moreover, \cite{Berestycki2014} provides  a  description of the asymptotic behaviour of these spectral quantities with respect to the scaling of the kernel. Since we strongly rely on these results, for the purpose of our analysis, we present a summary of these results in Section \ref{bcv-section-pre}.

Finally, we also want to stress that although we have a clear description of the existence and the non-existence of a positive solution for small $\eps$, the study of the convergence of $u_\eps$ as $\eps \to 0$ is quite delicate. Indeed, it is to be expected that the  
limiting solution will satisfy the problem   
 $$v(x)(a(x)-v(x))=0 \quad \text{in}\quad \R^N.$$
 But this problem has infinitely many bounded non negative solution in $L^{\infty}$.  E.g., for any set $Q\subset \R^N$ the function $a^+(x)\chi_{Q}$ is a solution. Therefore, owing to the lack of regularising effect of the dispersal operator,  we cannot rely on standard compactness results  to obtain a smooth limit.  If for the case $m=2$ we could rely on the elliptic regularity and the new description of Sobolev Spaces developed in \cite{Bourgain2001,Brezis2002,Ponce2004a, Ponce2004} to get some compactness, this characterisation does not allow us to treat the case $m<2$.   We believe that a new characterisation of Fractional Sobolev space in the spirit of the work of Bourgain, Brezis and Mironescu \cite{Bourgain2001,Brezis2002} will be helpful to resolve this issue.

 \medskip

The paper is organised as follows. In Section \ref{bcv-section-bio}, we discuss the biological interpretations of our results. There, we describe notion such as \textit{dispersal budget} and evolutionary stable strategies. In Section \ref{bcv-section-pre}, we recall some known properties and describe our recent  work on the principal eigenvalue $\lambda_p(\lb{\O}+a)$. We also  describe  the sharp persistence criteria for problem \eqref{bcv-eq-bounded} defined in a bounded domain $\O$ that are derived in terms of principal eigenvalues. 
In Sections \ref{bcv-section-crit} and \ref{bcv-section-lgt-beha}, we establish the sharp persistence criteria and prove the long time behaviour of the solution of \eqref{bcv-eq} (Theorem \ref{bcv-thm1}). We analyse the dependence of the persistence criteria  (Theorems \ref{bcv-thm3} and \ref{bcv-thm4}) in  Section \ref{bcv-section-asym}. Finally, in the  Section \ref{bcv-section-ext} we discuss the extension of the persistence criteria to kernels that are non longer assume to be  compactly supported. In the concluding section, we emphasize some of our results and corresponding biological interpretations. We also indicate several open problems and directions that arise naturally from this work.

\subsection{Notations}
To simplify the presentation of the proofs, we introduce some notations and various linear operator that we will use throughout this paper:
\begin{itemize}
\item $B_R(x_0)$  denotes the standard ball of radius $R$ centred at the point $x_0$
\item $\chi_R$ will always refer to the characteristic function of $B_R(0)$ .
\item $\s(\R^N)$ denotes the Schwartz space,\cite{Brezis2010} 
\item $C(\O)$ denotes the space of continuous function in $\O$, 
\item $C_0(\O)$ denotes the Banach space of continuous function in $\O$ that vanishes at the boundary.
\item For a positive integrable function $J\in \s(\R^N)$, the constant $\int_{\R^N}J(z)|z|^2\,dz$ will refer to 
$$\int_{\R^N}J(z)|z|^2\,dz:=\int_{\R^N}J(z)\left(\sum_{i=1}^Nz_i^2\right)\,dz$$
\item We denote by $\lb{\O}$ the continuous linear operator 
\begin{equation}\label{bcv-def-opl}
\begin{array}{rccl}
\lb{\O}:&C(\bar \O)&\to& C(\bar \O)\\
&u&\mapsto& \int_{\O}J(x-y)u(y)\,dy,
\end{array}
\end{equation}
where $\O\subset \R^N$.
\item $\lb{R}$ corresponds to  the continuous operator $\lb{\O}-I$ with $\O=B_R(0)$, ($I$ denotes the identity)
\item We will use  $\m$ to denote the operators $\lb{\O}-I$ with $\O=\R^N$ .
\item Finally,  $\m_\eps$ will denote the operator $\m$ with a rescaled kernel $\frac{1}{\eps^N}J\left(\frac{z}{\eps}\right)$
 and $\m_{\eps,m}:=\frac{1}{\eps^m}\m_\eps$
 \item To simplify the presentation of the proofs, we will also use the notation  $\beta(x):=\partial_sf(x,0)$.
\end{itemize}


\section{Ecological interpretation of our results}\label{bcv-section-bio}
 Here we define more precisely some concepts from ecology and discuss the ecological interpretation of our findings. 
We start with  the notion of \textit{dispersal budget} introduced in \cite{Hutson2003}. To this end, let us go back to the description of the dispersal of the population. The main ecological idea behind this ``\textit{dispersal budget}" is that
 we can consider that the amount of energy per individual that the organism can use to disperse is fixed (because of environmental or developmental constraints).  
 
 
Let us denote by $u$, the density of the population and suppose, as an example, that it represents a population of trees that produces and disperses its seeds.   Several dispersal strategies are then possible for this species: it can ``choose" to disperse few seeds  over long distances or produce many seed and  disperse them over a short distance and, of course, there are the intermediate strategies. We can then assume that the costs involved in the dispersal are proportional to
  \begin{itemize}
  \item the number of individuals dispersed, 
  \item a non decreasing and even function $\alpha$ of the distance moved.
  \end{itemize}
For a population of trees, for instance, the function $\alpha$ can be somehow related, to the amount of energy used to produce  seeds with  sophisticated shapes and components that allow it to  take advantage of wind or gravity  or a transport by animals.

If we discretize uniformly the space $\R^N$ by small cubes of volume $\delta x$ centered at  points $x_i$  and the time in time step $\delta t$,  then we can compute the  cost $C(x_i,x_j)$ associated to the transfer from a site $x_j$ to a site $x_i$ and  we get 
$$C(x_j,x_i)=J(x_j,x_i)\alpha(x_i-x_j)u(x_j,t)(\delta x)^{2} \delta t,$$
where the term  $J(x_j,x_i)u(x_j,t)(\delta x)^{2} \delta t$
is the total number of individuals that is transferred through a dispersal kernel $J$, from a site $x_j$ to a site $x_i$.    
The total cost in time $\delta t$ for a typical site $x_j$ is then: 

$$ u(x_j,t)\delta x\delta t \sum_{i}J(x_j,x_i)\alpha(x_i-x_j) \delta x . $$
Thus, if the amount of energy per individual is fixed, we get 
$$\sum_{i}J(x_i,x_j)\alpha(x_j-x_i) \delta x=c_0$$
where $c_0>0$ is a constant.  
Letting the volume of the cube $\delta x$ go down to $0$, we obtain
$$\int_{\R^N}J(x,y)\alpha(y-x)\,dx=c_0.$$
For an dispersal that only depends on the distance moved, (i.e. $J(x,y)=J(x-y)$), we get
$$\int_{\R^N}J(z)\alpha(z)\,dz=c_0.$$
To investigate the effects of the dispersal range on the persistence of the population, it is then reasonable to fix the cost function $\alpha$ and to allow the dispersal kernel $J$ to depend on a scaling factor, $\eps$.
Now, for a fixed cost function $\alpha(z)$ proportional to $|z|^m$, by taking a rescaled dispersal kernel of the form $\gamma(\eps,m)J_\eps(z):=\frac{\gamma(\eps,m)}{\eps^N}J\left(\frac{z}{\eps}\right)$ we get   
$$ \int_{\R^N}\gamma(\eps,m)J_\eps(z)|z|^m\,dz=c_0.$$ 
Thus,
$$ \gamma(\eps,m)=\frac{1}{\eps^m}\frac{c_0}{\int_{\R^N}J(z)|z|^m\,dz},$$
 and the rescaled dispersal kernel associated with the cost function $\alpha(z)$ proportional to $|y|^m$ is  
 $$ \frac{\alpha_0}{\eps^m} J_\eps(z),
 \quad \text{ with } \quad \alpha_0:=\frac{c_0}{\int_{\R^N}J(z)|z|^m\,dz}.$$
 
 
Fixing the dispersal budget now means that the dispersal process involved is defined by the operator $\M_{\eps,m}=\gamma(\eps,m)\m_\eps$ and that the species has the choice between large $\eps$, which corresponds to a strategy that  produces few offspring that are dispersed  far away  or small $\eps$, which corresponds to the opposite strategy, that is, producing a large number of offspring dispersed on  short  range. 
The rescaled dispersal kernel, will then depend upon three parameters, $\alpha_0$, $\eps$ and $m$. As explained in \cite{Hutson2003}, in the above setting, the constant $\frac{\alpha_0}{\eps^m}$ will refer to the \textit{rate} of dispersal, whereas $\eps$ is a measure of the \textit{range} of dispersal of the species. From the above formula, we clearly see how the cost function and the \textit{range} factor affect the \textit{rate} of dispersal.

The results of Theorems \ref{bcv-thm3}, \ref{bcv-thm4} and  \ref{bcv-thm5}  give some insight on the effects of the cost functions on the different strategies.    When $2>m>0$, we see that  strategies based on  sufficiently large   or sufficiently small \textit{range} factors enable the population  to persist. It is worth to mention that strategies based on  sufficiently large  \textit{range} factors always enable the population to persist when $m>0$. This is not necessarily true when $m=0$.  

To investigate further the effects of the \textit{dispersal budget} on  the different strategies, we can use the notion of 
Evolutionary Stable Strategy (ESS) introduced in Adaptive Dynamics, see \cite{Dieckmann1997,Diekmann2004,Metz1996,Vincent2005}.
  The concept of  ESS  comes from games theory and goes back to the work of Hamilton \cite{Hamilton1967} on the evolution of sex-ratio.  Roughly speaking, an Ecological Stable Strategy is a strategy such that if most of the members of a population adopt it, there is no ``mutant" strategy that would yield a  higher reproductive fitness. In this framework the strategies are compared using their relative pay-off.  This concept has been recently used  and adapted to investigate  ecological stable strategies of dispersal in several  contexts:  unconditional dispersal \cite{Dockery1998,Hastings1983,Hutson2001}, conditional dispersal \cite{Averill2012,Cantrell2006,Cantrell2008,Cantrell2009,Chen2008,Cosner2003,Cosner2005,Hambrock2009} and nonlocal dispersal \cite{Hutson2003,Kao2010}, see \cite{Cosner2014} for a review on this subject.

In these works, the general idea is to compare the dispersal  strategies  through the analysis of some invasion criteria.  
 Here,  following this idea, the strategies can be compared through the dynamics of a solution of a competitive system
\begin{align} 
&\partial_t u(t,x)=\m_{\eps_1,m}[u] + u(t,x)(a(x) -u(t,x) -v(t,x))\quad \text{ in }\quad \R^N\\
&\partial_t v(t,x)=\m_{\eps_2,m}[u] + v(t,x)(a(x) -u(t,x) -v(t,x))\quad \text{ in }\quad \R^N
\end{align}
where $u$ is a population that has adopted the spreading strategy $\eps_1$ and $v$ the $\eps_2$ strategy. The notion of ESS is then linked to some invasion condition  which is related to the stability of the equilibria $(u^*,0)$ where $u^*$ is a positive solution of the following problem 
$$\m_{\eps_1,m}[u] + u(x)(a(x) -u(x))=0\quad \text{ in }\quad \R^N.$$ 

The stability analysis of this equilibria, leads us to consider the sign of a principal eigenvalue of the operator $\m_{\eps_2,m} + a -u^*$. When $\lambda_p(\m_{\eps_2,m} + a -u^* )<0$  then the equilibria $(u^*,0)$ is unstable and a mutant may overtake the territory. Therefore, the strategy followed by $u$ will not be an ESS. On the contrary, when $\lambda_p(\m_{\eps_2,m} + a -u^* )>0$  the equilibria $(u^*,0)$ is stable and a mutant cannot invade its territory, making this strategy a possible candidate for an ESS.

In the context of local dispersion, this system has been introduced to discuss Ecological Stable Strategy of dispersal (see \cite{Averill2012,Cantrell2006,Cantrell2009,Chen2008,Cosner2003,Cosner2005,Dockery1998,Hastings1983,Hambrock2009}). The only difference is that, in this case, $\m_{eps_i,m}$ is an elliptic operator possibly involving an advection ($1^{st}$ order) term. 

 In this framework, two classes of dispersal strategies are distinguished: the unconditional dispersal vs the conditional dispersal.  Unconditional dispersal refers to dispersal without regard to the environment or the presence of other organisms. Pure diffusion and diffusion with physical advection (e.g. due to local climatic conditions) are examples of unconditional dispersal. Conditional dispersal refers to dispersal that is influenced by the environment or the presence of other organisms.
  
For particular conditional dispersal strategies, known as \textit{ideal free}, it is known that such strategies are evolutionary stable \cite{Cantrell2008,Cantrell2009,Cosner2005,Cosner2014}.  For long range dispersal, some nonlocal ideal free strategies have been recently exhibited \cite{Cosner2012}.

For unconditional dispersal strategies, within the framework of reaction diffusion models, it is known that the smaller disperser is always favoured \cite{Cantrell2009,Dockery1998,Hastings1983}. Such results are still valid for nonlocal dispersal strategies as soon as the cost function is a constant \cite{Hutson2003,Kao2010}. In such cases, the \textit{range} factor does not affect the \textit{rate} of dispersal.

For cost functions proportional to $|y|^m$ with $m>0$, the \textit{range} factor $\eps$ strongly  affects  the \textit{rate} of dispersal,  and the picture changes. From the asymptotics we have obtained (Theorems \ref{bcv-thm4} and \ref{bcv-thm5}), we  clearly see that the smaller $\textit{spreader}$ will not always be favoured. Indeed, consider two  species that have the same ecological niche  and suppose that  this ecological niche is bounded.

Let $\eps_1$ be the \textit{range} factor associated to one of the species and let us denote by $u^*$ the equilibrium reached by this population, i.e.  $u^*>0$ (or $u^*=0$ if there is no positive equilibrium) is the solution of $\m_{\eps_1,m}[u^*]+u^*(a(x)-u^*)=0$. Let $\eps_2$ be the \textit{range} factor associated with the other species. Let us look at the sign of $\lambda_p(\m_{\eps_2,m}+a -u^*)$. From our results, for $\eps_2$ large enough, we have 
 $ \lambda_p(\m_{\eps_2,m}+a -u^*)\approx -\sup_{\R^N}(a-u^*)$. Since $u^*$ satisfies $\m_{\eps_1,m}[u^*]+u^*(a(x)-u^*)=0,$  by the maximum principle, we infer that $\sup_{\R^N}(a-u^*)>0$. As a consequence, for $\eps_2$ large enough,  $\lambda_p(\m_{\eps_2,m}+a -u^*)<0$ and the population having the \textit{range} factor $\eps_1$ will be wiped out. In other words, if a competing species disperses on a sufficiently long range, it will invade the territory occupied previously by the species of \textit{range} factor $\eps_1$. Conversely, when the cost functions is sub-quadradic, i.e. $m<2$,    for $\eps_2$ small enough,  we  have also  $ \lambda_p(\m_{\eps_2,m}+a -u^*)\approx -\sup_{\R^N}(a-u^*)$.
 This means that a smaller spreading factor can also lead to invasion. For a cost function proportional to $|y|^m$ with $0<m<2$, the effect of the cost function is then  twofold, both large and small \textit{spreader} species can be favoured.

From our result, we  also infer  that within the framework of the space of  a quadratic cost function $(m=2)$, the ubiquity strategy  ($\eps=\infty$) seems to be an ESS. Indeed, for this case,  we are led to consider the sign of $\lambda_p(\m_{\eps_2,2} + a -a^+ )$ which is positive for any $\eps>0$. It is worth noticing that in this situation, the smallest spreader ($\eps =0$) is never an ESS. For such  singular strategy, what matter is the sign of $\lambda_p(\m_{\eps_2,2} + a -u^* )$ where $u^*$ is the solution of 
$$\frac{D_2(J)}{2N}\Delta u^* +u^*(a(x)-u^*)=0.$$
For $\eps_2$ large enough, we see that $\lambda_p(\m_{\eps_2,2} + a -u^* )<0$. Thus the equilibrium $(u^*,0)$ is unstable, rendering the singular strategy not evolutionary stable.

Such behaviour  stands  in contrast with known results on ESS strategy  governed by the rate of dispersion \cite{Hutson2003,Kao2010}. These properties show that nonlocal diffusion exhibits quite different behaviour with respect to what was known in the local diffusion case, where in such case the slowest possible rate  is always the best strategy.  




\section{Preliminaries}\label{bcv-section-pre}

 In this section, we recall some known results on the principal eigenvalue of a linear non-local operator $\lb{\O}+a$  and on  the nonlocal Fisher-KPP equation :
 \begin{equation}\label{bcv-eq-kpp}
 \frac{\partial u}{\partial t}(t,x)=\oplb{u}{\O}+f(x,u(t,x)) \quad \text{ in }\quad \R^+\times\O,
 \end{equation}
 considered in a bounded domain  $\O\subset \R^N$.

\subsection{Principal eigenvalue for non-local operators}
In this subsection, we focus on the properties of the spectral problem 

\begin{equation}\label{bcv-eq-pev}
\oplb{\varphi}{\O}+a\varphi+\lambda \varphi=0 \quad\text{ in }\quad \O.
\end{equation}

In contrast with elliptic operators, when $a$ is not a constant,  neither $\lb{\O} +a+\lambda$ nor its inverse  are  compact operators and the description of the spectrum of $\lb{\O} +a$  using  the Krein-Rutman Theory fails. However as shown in \cite{Coville2010},  some variational formula introduced in  \cite{Berestycki1994} to characterise  the first eigenvalue of elliptic operators $\E:=a_{ij}(x)\partial_{ij}+b_i(x)\partial_i +c(x)$,
\begin{equation}\label{bcv-eq-lambda1-form-sup}
 \lambda_1 (\E) := \sup \{\lambda \in \R \, |\, \exists \, \varphi \in W^{2,n}(\O), \varphi > 0 \;\text{ so that }\; \ope{\varphi} + \lambda\varphi\le 0\},
 \end{equation}
can be transposed to the operator $\lb{\O}+a$. Namely, we define the quantity

\begin{equation}\label{bcv-eq-lambda-form-sup}
 \lambda_p (\lb{\O} + a) := \sup \{\lambda \in \R \, |\, \exists \, \varphi \in C(\O), \varphi > 0 \;\text{ so that }\; \oplb{\varphi}{\O}+ a\varphi + \lambda\varphi\le 0\}.
 \end{equation}
 $\lambda_p(\lb{\O} +a)$ is well defined and we call it the generalised principal eigenvalue. 
 
 As in \cite{Coville2010},   the quantity  defined by \eqref{bcv-eq-lambda-form-sup} is  not always  an eigenvalue of $\lb{\O}+a$ in a reasonable Banach space.  This means that there is not always a positive continuous eigenfunction associated with $\lambda_p$. This stands in contrast with elliptic PDE's.   
However, as proved in \cite{Coville2010,Kao2010,Shen2012}, when $\O$ is a bounded domain we can give some conditions on the coefficients that guarantee  the existence of a positive continuous eigenfunction. For example,  if   the function $a$ satisfies
 $$\frac{1}{\sup_{\O}a -a}\not\in L_{loc}^1 (\bar \O ),  $$
 then $\lambda_p(\lb{\O}+a)$ is an eigenvalue of $\lb{\O}+a$ in  the Banach space $C(\bar\O)$ that is, it is associated to a positive continuous eigenfunction, continuous up to the boundary. 

Another useful criteria  that guarantees the  existence of a continuous principal eigenfunction is  
\begin{proposition} \label{bcv-prop-phip}
 Let $\O$ be a bounded domain and let $\lb{\O}$ be as in \eqref{bcv-def-opl} then there exists a positive continuous eigenfunction associated to 
$\lambda_p$ if and only if $\lambda_p(\lb{\O}+a(x))< -\sup_{\O}a$. 
 \end{proposition}
 A proof of this proposition can be found for example in \cite{Coville2013,2013arXiv1305.7122C}. To have a more complete description of the properties of $\lambda_p$ in bounded domains see   \cite{Coville2013a}.

  Next, we recall some properties of  $\lambda_p$ that we  constantly use throughout this paper:
 \begin{proposition}\label{bcv-prop-pev}
\begin{itemize}
\item[(i)] Assume $\O_1\subset\O_2$, then for the two operators $\lb{\O_1}+a$ and $\lb{\O_2}+a$
respectively defined on $C(\O_1)$ and $C(\O_2)$, we have :
 $$
\lambda_p(\lb{\O_1}+a)\ge \lambda_p(\lb{\O_2}+a).
$$
\item[(ii)]For a fixed $\O$  and assume that $a_1(x)\ge a_2(x)$, for all $x \in \O$. Then  
$$
\lambda_p(\lb{\O}+a_2)\ge\lambda_p(\lb{\O}+a_1).
$$

\item[(iii)] $\lambda_p(\lb{\O}+a)$ is Lipschitz continuous with respect to  $a$. More precisely,
$$|\lambda_p(\lb{\O}+a)- \lambda_p(\lb{\O}+b)|\le \|a-b\|_{\infty}$$

\item[(iv)]  The following estimate always holds
$$-\sup_{\O}\left(a(x)+\int_{\O}J(x-y)\,dy\right)\le \lambda_p(\lb{\O}+a)\le -\sup_{\O}a.$$
\end{itemize}
\end{proposition}
We refer to \cite{Coville2010,2013arXiv1305.7122C} for the proofs of $(i)-(iv)$. Let us also recall the two following  results proved in \cite{Berestycki2014}.
\begin{lemma}\label{bcv-lem-lim}
 Assume that $a$ achieves its maximum in $\O$ and let  $\lb{\O}+a$ be defined as in \eqref{bcv-def-opl} with $J$ satisfying $(H1-H2)$. Assume further that $J$ is compactly supported.
Let $(\O_n)_{n \in \R}$ be a sequence of subset of $\O$ such that $\lim_{n\to \infty}\O_n =\O$, $\O_n \subset \O_{n+1}$.
Then, we have 
$$ \lim_{n\to \infty}\lambda_p(\lb{\O_n}+a)=\lambda_p(\lb{\O}+a)$$ 
\end{lemma}

\begin{lemma}\label{bcv-lem-scal-eq}
Assume  that $a(x)\in C(\R^N)\cap L^{\infty}(\R^N)$. Then for all $\eps>0$ 
 $$\lambda_p(\M+a)=\lambda_p(\M_{\eps}+a_\eps),$$
where $a_\eps(x):=a\left(\frac{x}{\eps}\right)$ and $\M_\eps[\varphi](x):=\frac{1}{\eps^N}\int_{\R^N}J\left(\frac{x-y}{\eps}\right)\varphi(y)\,dy -\varphi(x)$.
\end{lemma}

\bigskip

Lastly, we recall some recent results obtained in \cite{Berestycki2014} on the characterisation of the generalised principal eigenvalue $\lambda_p(\m_{\eps,m} +a)$. Motivated by the works \cite{Berestycki2006, Berestycki2007, Berestycki2008} on the generalised principal eigenvalue of an elliptic operators, let us introduce the two  definitions : 

\begin{definition} Let $\lb{\O}+a$ be  as in \eqref{bcv-def-opl}.  We define the following quantities:
\begin{align}
&\lambda_p'(\lb{\O}+a)&:=&\inf\{\lambda\in\R \,|\, \exists \varphi\ge 0, \varphi\in C(\O)\cap L^\infty(\O), \textrm{ such that } \oplb{\varphi}{\O}+(a+\lambda)\varphi\geq 0\;\text{ in}\; \O\}\label{bcv-eq-def-lambda'},\\
&\lambda_v(\lb{\O}+a)&:=&\inf_{\varphi \in L^2(\O),\varphi\not \equiv 0} -\frac{\langle\oplb{\varphi}{\O}+a\varphi, \varphi \rangle}{\langle\varphi, \varphi \rangle},\\
&&=&\inf_{\varphi \in L^2(\O),\varphi\not \equiv 0} \frac{\int_{\O}\int_{\O}J(x-y)(\varphi(x)-\varphi(y))^2\,dxdy -\int_{\O}(a-1+k(x))\varphi^2(x)\,dx}{\|\varphi\|_{L^2(\O)}^2}
\label{bcv-eq-def-lambda-v}.
\end{align}
where $k(x):=\int_{\O}J(y-x)\,dy$ and $\langle \cdot, \cdot\rangle$ denotes the standard scalar product in $L^{2}(\O)$.
\end{definition}

These definitions  are natural extension  of the definitions  known for an elliptic operator. It is worth to mention that those definitions  have already been used for the study of \eqref{bcv-eq} in several papers \cite{Coville2008b,Coville2013,2013arXiv1305.7122C,Garcia-Melian2009a,Ignat2012}, but the relation between $\lambda_p,\lambda_p'$ and $\lambda_v$  had not been clarified.

For elliptic operators, the analogues of these three quantities are equivalent on bounded domain \cite{Berestycki1994}. This is not necessarily the case for unbounded domains, where examples can be constructed \cite{Berestycki2006,Berestycki2007,Berestycki2010},  for which  $\lambda_1>\lambda_1'$. Since the operator, $\lb{\O}+a$, shares many properties with elliptic operators, it is  suspected that the three quantities, $\lambda_p,\lambda_p'$ and $\lambda_v$, are not necessarily  equal.  However, for compactly supported kernel ,?, $J$, we have:

 \begin{theorem}[\cite{Berestycki2014}]\label{bcv-thm6}
 Let $J$ be compactly supported satisfying {\rm (H1)--(H2)}. Assume that $a\in C(\R^N)\cap L^{\infty}(\R^N)$. Then we have 
 $$\lambda_p(\m_{\eps,m} +a)=\lambda_p'(\m_{\eps,m} +a)=\lambda_v(\m_{\eps,m} +a).$$
 Moreover, we have the following asymptotic behaviour as either $\eps \to 0$ or $\eps\to \infty$ :
  \begin{itemize}
\item When $0<m\le 2 \qquad\lim_{\eps\to +\infty}\lambda_p(\m_{\eps,m}+a)=-\sup_{x\in\R^N}a(x)$
\item When $m=0, \qquad\lim_{\eps\to +\infty}\lambda_p(\m_{\eps}+a)=1-\sup_{x\in\R^N}a(x)$
\item When $0\le m <2, \qquad\lim_{\eps\to 0}\lambda_p(\m_{\eps,m}+a)=-\sup_{x\in\R^N}a(x)$ 
\item When $m=2$ and $a$ is globally Lipschitz, then 
$$\lim_{\eps\to 0}\lambda_p(\m_{\eps,2}+a)=\lambda_1\left(\frac{D_2(J)}{2N}\Delta +a\right),$$ 
where we recall that $$D_2(J):=\int_{\R^N}J(z)z^2\,dz.$$

and $$\lambda_1\left(\frac{D_2(J)}{2N}\Delta +a\right):=\inf_{\varphi \in H^1_0(\R^N),\varphi\not \equiv 0} \frac{D_2(J)}{2N}\frac{\int_{\R^N}|\nabla\varphi|^2(x)\,dx}{\nlto{\varphi}^2} -\frac{\int_{\R^N}a(x)\varphi^2(x)\,dx}{\nlto{\varphi}^2}. $$
\end{itemize} 
 
 \end{theorem}
\bigskip

A similar result also holds for the rescaled operator $$\lb{R,\eps,m}:=\frac{1}{\eps^m}\lb{R,\eps}$$ with 
$\lb{R,\eps}$ defined by :
$$\lb{R,\eps}[\varphi](x)=\int_{B_R}J_\eps(x-y)\varphi(y)\,dy -\varphi(x)$$
 with the rescaled kernel $J_\eps$. Namely,
\begin{theorem}[\cite{Berestycki2014}]\label{bcv-thm7}
 Assume $J$  satisfies {\rm (H1)--(H2)} and let  $a(x)\in C(\bar B_R(0))$. Then we have 
 $$\lambda_p(\lb{R,\eps,m} +a)=\lambda_p'(\lb{R,\eps,m} +a)=\lambda_v(\lb{R,\eps,m} +a).$$
  Moreover, we have the following asymptotic behaviour:
   \begin{itemize}
\item When $0<m\le 2 \qquad\lim_{\eps\to +\infty}\lambda_p(\lb{R,\eps,m}+a)=-\sup_{B_R(0)}a(x)$
\item When $m=0, \qquad\lim_{\eps\to +\infty}\lambda_p(\lb{R,\eps}+a)=1-\sup_{B_R(0)}a(x)$
\item When $0\le m <2, \qquad\lim_{\eps\to 0}\lambda_p(\lb{R,\eps,m}+a)=-\sup_{B_R(0)}a(x)$ 
\item When $m=2$ and assuming that the function $a$ is globally Lipschitz, we get
$$\lim_{\eps\to 0}\lambda_p(\lb{R,\eps,2}+a)=\lambda_1\left(\frac{D_2(J)}{2N}\Delta +a, B_R(0)\right)$$ 
\end{itemize}

 \end{theorem}
\bigskip

\subsection{Existence criteria for the KPP-equation \eqref{bcv-eq-kpp}}

With this notion of generalised principal eigenvalue, it has been shown \cite{Bates2007,Coville2010} that on bounded domains, the existence of a positive stationary solution of \eqref{bcv-eq-kpp} is conditioned by the sign of $\lambda_p(\lb{\O}+\partial_sf (x,0))$. More precisely,
 \begin{theorem}[\cite{Bates2007,Coville2010}]\label{bcv-thm-existence} 
 Let $\O$ be a bounded domain and  $\lb{\O}$ defined as in \eqref{bcv-def-opl}. Assume that $f$ satisfies $(H3)$. Then there exists a unique positive continuous function, $\bar u$, stationary solution  of \eqref{bcv-eq-kpp} if and only if $\lambda_p(\lb{\O}+ \partial_s f (x,0)) < 0$. Moreover,  if   $\lambda_p \ge 0$ then $0$ is the only  non negative bounded stationary solution  of \eqref{bcv-eq-kpp}. In addition, for any positive  continuous solutions of \eqref{bcv-eq-kpp} we have the following dynamics :
 \begin{itemize}
 \item[(i)] When $\lambda_p\ge 0$, $$\lim_{t\to \infty}u(t,x)\to 0 \quad \text{ uniformly in } \O,$$  
 \item[(ii)] When $\lambda_p< 0$, $$\lim_{t\to \infty}u(t,x)\to \bar u \quad \text{ uniformly in } \O.$$
 \end{itemize}  
 \end{theorem}
 
 \begin{remark}
This  existence criteria is similar to those known for the PDE reaction diffusion versions of \eqref{bcv-eq-kpp} \cite{Berestycki2005,Cantrell1991,Cantrell1998,Englander1999}.
 \end{remark}

\section{Existence/non existence  and  uniqueness of a non-trivial solution }\label{bcv-section-crit}

In this section we construct a non-trivial solution of \eqref{bcv-eq} and prove the necessary and sufficient condition stated in Theorem \ref{bcv-thm1}. We treat successively the  existence of a solution,  its uniqueness and non-existence. 

\subsection{Existence of a non-trivial positive solution}\label{bcv-ss-existence}
 The construction follows a basic approximation scheme previously used for example in \cite{Berestycki2009}.
 We introduce the following approximated problem : 
 \begin{equation}\label{bcv-eq-approx}
 \oplr{u}{R}+f(x,u)=0 \quad \text{ in }\quad \bar B(0,R)
 \end{equation}
where $B(0,R)$ denotes the ball of radius $R$ centred at the origin. By Theorem \ref{bcv-thm-existence}, for any $R>0$ the existence of a unique positive solution of  \eqref{bcv-eq-approx} is conditioned by the sign of $\lambda_p(\lb{R}+ \beta)$ where $\beta(x):=\partial_uf(x,0)$.
 Since $$\lim_{R\to +\infty} \lambda_p(\lb{R}+\beta)=\lambda_p(\M+\beta)<0,$$
by Lemma \ref{bcv-lem-lim} there exists $R_0>0$ such that $$ \forall R\ge R_0, \lambda_p(\lb{R}+\beta)<0.$$ 
As a consequence, by Theorem \ref{bcv-thm-existence}, for all $R>R_0$ there exists a unique positive solution of \eqref{bcv-eq-approx} that we denote $u_R$. Moreover, since for all $R>0, \sup_{B_R(0)}S(x)$ is a super-solution of  \eqref{bcv-eq-approx}, by a standard sweeping argument since the solution to \eqref{bcv-eq-approx} is unique, we get
$$ \forall R>0,\, u_R\le \sup_{B_R(0)}S(x) \quad \text{in}\quad B(0,R).$$ 
 
On another hand, for any $R_1>R_2$, the solution $u_{R_1}$ is a super-solution for the problem  
 
\begin{equation}\label{bcv-eq-approx0}
 \oplr{u}{R_2}+f(x,u)=0 \quad \text{ in }\quad  B(0,R_2)
 \end{equation}
By the same  sweeping argument we get
$$  u_{R_2}(x)\le u_{R_1}(x) \quad \text{in}\quad B(0,R_2).$$
Thus, the map $R\mapsto u_R$ is monotone increasing. 

 The idea is to obtain a positive solution of \eqref{bcv-eq} as a limit of the  positive solution of \eqref{bcv-eq-approx}.
To this end we construct a uniform super-solution of  problem \eqref{bcv-eq}.

 \begin{lemma}\label{bcv-lem-unif-supersol} There exists $\bar u \in C_0(\R^N)\cap L^1(\R^N)$, $\bar u>0$ such that  $\bar u$ is a super-solution of  problem \eqref{bcv-eq}. 
\end{lemma}

\begin{proof}
Let us fix $\nu>0$ and $R_0>1$ so that  $\nu<- \limsup_{|x|\to\infty} \beta(x)$ and  $\beta(x)\le -\frac{\nu}{2}$ for all $|x|\ge R_0$.
Consider the function  $$w(x)=C e^{-\alpha|x|},$$
where $C$ and $\alpha$ are to be chosen.  
For all $x\in \R^N\setminus B_{R_0}(0)$ we get:
\begin{align*}
 \opm{w}(x)+\beta(x)w(x)&=Ce^{-\alpha|x|}\left(\int_{\R^N}J(x-y)e^{-\alpha(|y|-|x|)}\,dy-1+\beta(x)\right),\\
&\le w(x)\left(\int_{\R^N}J(z)e^{\alpha(|z|)}\,dz -1 -\frac{\nu}{2}\right).
\end{align*}
Therefore, $w$ satisfies
\begin{equation}
\opm{w}(x)+\beta(x)w(x) \le h(\alpha)w(x) \quad \text{ in }\quad \R^N \setminus B_{R_0}(0),
\end{equation}
where $h(\alpha)$ is defined by
$$h(\alpha)=-1 -\frac{\nu}{2}.$$

Since $J$ is compactly supported, thanks to  Lebesgue's Theorems, we can check that  $h(\cdot)$  is a smooth ($C^2$) convex increasing function of $\alpha$. Moreover, we have 
$$\lim_{\alpha\to 0}h(\alpha)=h(0)=  -\frac{\nu}{2}.$$
Therefore, by continuity of $h$, we can choose $\alpha$ small enough such that $h(\alpha)<0$.
For such an $\alpha$, we get 
\begin{equation}\label{bcv-eq-w}
\opm{w}(x)+\beta(x)w(x) \le h(\alpha)w(x)<0 \quad \text{ in }\quad \R^N \setminus B_{R_0}(0).
\end{equation}

Let $M:=\sup_{B_{2R_0}(0)}S(x)$ and  let us fix  $C=2Me^{2\alpha R_0}$. We consider now the continuous function 
$$\bar u (x):= \begin{cases} C e^{-\alpha|x|} \quad \text{ in } \quad \R^N\setminus B_{2R_0}(0),\\
2M \quad \text{ in } \quad B_{2R_0}(0).
 \end{cases}
 $$

By direct computation we can check that $\bar u$ is a super-solution of the problem \eqref{bcv-eq}. Indeed, for any $x \in  B_{2R_0}(0)$, we have $\bar u =2M> \sup_{B_{2R_0}(0)} S(x)$  which implies that  $f(x,\bar u)=f(x,2M)\le 0$ and thus  
$$\opm{\bar u}(x) +f(x,\bar u(x))\le 2M\int_{\R^N}J(x-y)\,dy -2M +f(x,2M)\le f(x,2M)\le 0.$$
Then, for  $x\in \R^N\setminus B_{2R_0}(0)\subset\R^N\setminus B_{R_0}(0) $  by \eqref{bcv-eq-w} we have 
\begin{align*}
 \opm{\bar u}(x) +f(x,\bar u(x))\le \opm{\bar u}(x) +\beta(x)w(x) &\le \opm{w}(x) +\beta(x)w(x),\\
 & \le h(\alpha)w(x)\le 0. 
 \end{align*}

\end{proof}

We are now in a position to construct a positive solution of \eqref{bcv-eq}.  By Lemma \ref{bcv-lem-unif-supersol}, there exists $\bar u$  a positive continuous super-solution of  problem \eqref{bcv-eq}. Therefore, for any $R>0$, $\bar u$ is also a positive continuous super-solution of the problem \eqref{bcv-eq-approx}.
Therefore, by the standard sweeping argument, we can check that for all $R\ge R_0$ the unique positive continuous solution of \eqref{bcv-eq-approx} satisfies $u_R\le \bar u$ in $B_R(0)$.  By letting  $R\to\infty$ and observing that $u_R \in C(B_R)$ is  uniformly bounded and monotone with respect to $R$,  we get $u_R\to \tilde u:=\lim_{R\to\infty}u_R$. The function $\tilde u$ is a non-negative solution of \eqref{bcv-eq} and it  is non zero since $0\le \tilde u \le \bar u $ and $$0<u_{R} \le \tilde u  \quad \text{ in } B(0,R),\quad \text{ for all } R\ge R_0.$$


\subsection{Uniqueness}\label{bcv-ssection-uniq}

Having constructed a  positive solution of \eqref{bcv-eq} in $L^1(\R^N)$, we now prove its  uniqueness. 
Assume by contradiction that  $v\in C(\R^N)\cap L^{\infty}(\R^N)$ is another positive solution. Then $v$ is a supersolution of problem \eqref{bcv-eq-approx} for any $R>0$. Therefore $v\ge u_R$ for all $R\ge R_0$. Since $u_R$ is monotone with respect to $R$, it follows that 
$v\ge \tilde u:=\lim_R u_R(x).$   
By assumption  $v\not\equiv \tilde u$ almost everywhere.  Recall that the functions $v$ and $\tilde u$ satisfies:
\begin{align}
\opm{\tilde u} +f(x,\tilde u)=0 \quad \text{ in }\quad \R^N, \label{bcv-eq-tildeu}\\
\opm{v} +f(x, v)=0 \quad \text{ in }\quad \R^N. \label{bcv-eq-v}
\end{align} 
Multiplying \eqref{bcv-eq-tildeu} by $v$ and \eqref{bcv-eq-v} by $u$ we get after integration over $\R^N$ :
\begin{align}
\int_{\R^N}\int_{\R^N}J(x-y)\tilde u(y)v(x)\,dydx - \int_{\R^N}\tilde u(x)v(x)\,dx +\int_{\R^N}v(x)f(x,\tilde u(x))\,dx=0, \label{bcv-eq-uniq1} \\
\int_{\R^N}\int_{\R^N}J(x-y)\tilde u(x)v(y)\,dydx - \int_{\R^N}\tilde u(x)v(x)\,dx +\int_{\R^N}\tilde u(x)f(x,v(x))\,dx=0.\label{bcv-eq-uniq2}
\end{align}
 Subtracting \eqref{bcv-eq-uniq2} from \eqref{bcv-eq-uniq1} yields 
 
 $$0<\int_{\R^N}v(x)\tilde u \left[\frac{f(x,\tilde u(x))}{\tilde u(x)}-\frac{f(x,v(x))}{v(x)}\right]\,dx=0, $$ 
 which is a contradiction since $\tilde u \le v $ and $f(x,s)/s$  is decreasing.

When $v$ is just a $L^{\infty}$ solution of \eqref{bcv-eq},  a similar argument holds using an adapted version of the  maximum principle (Theorem 1.4 in \cite{Coville2008b}).

 \subsection{Non-existence of a solution}
 In this section, we deal with the non-existence of positive solution when $\lambda_p(\M+\beta)\ge 0$. To simplify the presentation of the proofs, we treat the two cases:  $\lambda_p(\M+\beta> 0$ and $\lambda_p(\M+\beta)= 0$ separately, the proof in the second case being more involved.
 \subsubsection*{Case $\lambda_p(\M+\beta)> 0$:}
 In this situation we argue as follows. Assume by contradiction that a positive bounded solution $u$ exists.
 By assumption, $u$ satisfies  
\begin{equation}\label{bcv-eq-linear-soussol}
\opm{u}(x)+\beta(x)u(x)\ge 0.
\end{equation}
Therefore, $u$ is a test function for $\lambda_p'(\M+\beta)$ and we get 
$\lambda_p'(\M+\beta)\le 0.$
Since by Theorem  \ref{bcv-thm6}, $\lambda_p(\M+\beta)\le \lambda_p'(\M+\beta)$ we get an obvious contradiction.
\subsubsection*{Case $\lambda_p(\M+\beta)=0$:}
We argue again by contradiction. 
Assume  that  a non-negative, non identically zero, bounded solution $u$ exists. By a straightforward application of the maximum principle, since $u\not\equiv 0$ we have  $u>0$ in $\R^N$. 
By the above argument we have $\lambda_p(\M+\beta)=0=\lambda_p'(\M+\beta)$ and by $(iv)$  of Proposition \eqref{bcv-prop-pev} we get  the following estimate :
\begin{equation}\label{bcv-eq-gamma0}
\sup_{\R^N}(\beta(x)-1)\le 0.
\end{equation} 
Let us denote $\gamma(x):=\frac{f(x,u(x))}{u(x)}$, then we obviously have 
\begin{equation}
\label{bcv-eq-gamma1}
J\star u(x) -u(x) +\gamma(x) u(x)=0 \quad \text{ in }\quad \R^N
\end{equation}
By definition of $\lambda_p'$ we have $\lambda_p'(\M+\gamma)\le 0$.
By construction, $\gamma(x)\le \beta(x)$, so by combining \eqref{bcv-eq-gamma1}  with  Proposition \ref{bcv-prop-pev}, Theorem  \ref{bcv-thm6} and the definition of $\lambda_p(\M+\gamma)$ we can infer that 
 $$\lambda_p(\M+\gamma)\le \lambda_p'(\M+\gamma)\le 0\le \lambda_p(\M+\beta)\le \lambda_p(\M+\gamma).$$
Hence $ \lambda_p(\M+\gamma)=0$. 
Let us denote $\eta \in C(\R^N)$  a smooth regularisation of $\chi_{B_1(0)}$ the characteristic function of the unit  ball. Since 
$\gamma(x)<\beta(x)$  in $\R^N$, we can find $\eps_0>0$ small enough so that for all $\eps\le \eps_0$ 
$$\gamma(x)\le\gamma(x)+\eps\eta(x)<\beta(x) \quad \text{ in }\quad \R^N.$$
 
By $(i)$ of Proposition \ref{bcv-prop-pev}, we then have 
$$0=\lambda_p(\M+\beta)\le \lambda_p(\M+\gamma +\eps \eta)\le \lambda_p(\M+\gamma)=0.  $$ 
Next we claim that 
\begin{lemma}
There exists $R_1>0$ and  $\psi>0$, $\psi \in C(\R^N)\cap L^1(\R^N)$ such that 
$$J\star \psi(x)-\psi(x)+ (\gamma(x)+\eps \eta(x))\psi(x)= 0 \quad \text{ in }\quad \R^N. $$
\end{lemma} 
Assume for a moment that the Lemma holds. Then by arguing as in  subsection \eqref{bcv-ssection-uniq}, since $\psi \in L^1$ we get the following contradiction 
$$0=-\eps\int_{\R^N}u(x)\psi(x)\eta(x)\,dx<0.$$

 \begin{proof}[Proof of the Lemma]

For convenience we denote $\tilde \gamma:=\gamma +\eps \eta$.
By \eqref{bcv-eq-gamma0}, since $\tilde \gamma< \beta$ we also have
\begin{equation}\label{bcv-eq-gamma2}
0< - \sup_{\R^N}(\tilde \gamma-1).
\end{equation} 
From this inequality,  by using Proposition \ref{bcv-prop-phip}  and Lemma \ref{bcv-lem-lim}  we see that there exists $R_0$ such that for all $R\ge R_0$ there exists a positive eigenfunction $\varphi_R\in C(\bar B(0,R))$ associated with the generalised principal eigenvalue $\lambda_p(\lb{R}+\tilde \gamma)$ of the approximated problem
\begin{equation}\label{bcv-eq-tilde-gamma-phir}
\oplr{\varphi}{R} +(\tilde\gamma+\lambda)\varphi=0 \quad \text{ in }\quad B(0,R)
\end{equation}

Consider now the increasing sequence  $(R_n)_{n\in\N}:=(R_0+n)_{n\in\N}$  and let $(\varphi_n)_{n\in \N}$ be the sequence of  positive principal eigenfunction associated with $\lambda_p(\lb{R_n}+\tilde \gamma)$. Without loss of generality, we can assume that for all $n$, $\varphi_n(0)=1$. 

For all $n$, $\varphi_n$ satisfies
\begin{equation}\label{bcv-eq-lim-lpn3}
\oplb{\varphi_n}{R_n}+(\tilde\gamma+\lambda_p(\lb{R_n}+\tilde \gamma))\varphi_n=0 \quad \text{in}\quad B_{R_n}.
\end{equation}

Let us now define   $b_n(x):=-\lambda_{p}(\lb{R_n}+\tilde \gamma)-\tilde \gamma(x)$. Then $\varphi_{n}$ satisfies  
\begin{equation*}\label{bcv-eq-lim-lpn4}
  \oplb{\varphi_{n}}{R_n}=b_n\varphi_{n} \quad \text{in}\quad B_{R_n}.
  \end{equation*}
 By construction, for all $n\ge 0$ we have $b_n\ge-\lambda_{p}(\lb{R_{n_0}}+\tilde\gamma)-\sup_{\R^N}(\tilde \gamma(x)-1)>0 $,  therefore the Harnack inequality (Theorem 1.4  in \cite{Coville2012}) applies to $\varphi_{n}$. Thus for $n\ge 0$ fixed and for all compact set $\o \subset \subset B_{R_n}$ there exists a constant $C_n(\o)$ such that 
$$\varphi_{n}(x)\le C_n(\o)\varphi_{n}(y) \quad \forall \, x,y \in \o.$$

Moreover, the constant $C_n(\o)$ only depends  on $\bigcup_{x\in \o}B_{r_0}(x)$ and is  monotone decreasing  with respect to  
$\inf_{x\in B_{R_n}}b_n(x)$.
For all $n\ge 0$, the function $b_n(x)$ being uniformly bounded from below by a constant independent of $n$, we can choose the constant $C_n$ so that it is bounded from above independently of $n$ by a constant $C(\o)$.  Thus,  
$$\varphi_{n}(x)\le C(\o)\varphi_{n}(y) \quad \forall \, x,y \in \o.$$ 

From  the normalization $\varphi_{n}(0)=1$, we infer that  the sequence $(\varphi_{n})_{n\ge 0}$  is locally uniformly bounded in $\R^N$. Moreover, from a standard diagonal extraction argument, there exists a subsequence still denoted $(\varphi_{n})_{n\ge  0}$ such that $(\varphi_{n})_{n\ge 0}$ converges  locally uniformly  to  a continuous function  $\varphi$. Furthermore,  $\varphi$  is a non-negative non trivial function and $\varphi(0)=1$.

Since $J$ is compactly supported, we can pass to the limit in  equation \eqref{bcv-eq-lim-lpn3} using the Lebesgue  monotone convergence theorem   and  get 
$$ \opm{\varphi}+(\tilde\gamma+\lambda_p(\m+\tilde \gamma))\varphi = 0 \quad \text{ in }\quad \R^N.$$
Hence, we have 
%
%
    
\begin{equation}\label{bcv-eq-tilde-gamma-phi}
\opm{\varphi} +\tilde\gamma\varphi=0 \quad \text{ in }\quad \R^N.
\end{equation}

%
%
%
%
To conclude the proof of this Lemma, we characterise the behaviour of $\varphi(x)$ for $|x|>>1$.\\
Let us denote $0<\nu<-\limsup_{|x|\to \infty}\beta(x)$ and let us fix $R_1$ so that $\beta(x)\le -\frac{\nu}{2}$ for $|x|\ge R_1$.

Since by Lemma \ref{bcv-lem-lim}, $\lambda_p(\lb{R}+\tilde \gamma)\to \lambda_p(\M+\tilde \gamma)=0$, we can take  $R_1$ larger  if necessary to achieve  
$$\tilde \gamma(x)+\lambda_p(\lb{R}+\tilde \gamma)\le -\frac{\nu}{4} \quad \text{for}\quad |x|\ge R_1.$$

Let us now consider $\psi(x):=Ce^{-\alpha(|x|-R_1)}$ where $C$ and $\alpha$ will be chosen later on.
By a straightforward computation, we  see that for all  $R>R_1$

\begin{align*}
\oplr{\psi}{R}(x)+(\tilde\gamma(x)+\lambda_p(\lb{R}+\tilde \gamma))\psi(x)&\le \psi(x)\left(\int_{\R^N}J(z)e^{\alpha|z|}dz-1-\frac{\nu}{4}\right)\quad \text{ for }\quad |x|\ge R_1,\\
&\le h(\alpha)\psi(x)\quad \text{ for }\quad |x|\ge R_1,
\end{align*} 
with $$ h(\alpha):=\left(\int_{\R^N}J(z)e^{\alpha|z|}dz-1-\frac{\nu}{4}\right).$$
Since $J$ is compactly supported, by the Lebesgue Theorem, the function $h$ is  continuous  and $h(0)=-\frac{\nu}{4}$.
By assumption $\nu>0$, and by continuity of $h$ there exists $\alpha_0>0$ such that $h(\alpha_0)<0$.
Thus,  for $\alpha=\alpha_0$ we achieve 
\begin{equation}\label{bcv-eq-tilde-gamma-supersol}
\oplr{\psi}{R}(x)+(\tilde\gamma(x)+\lambda_p(\lb{R}+\tilde \gamma(x)))\psi(x)\le 0\quad \text{ for }\quad |x|\ge R_1.
\end{equation}

Recall that by construction, the function  $\varphi_n$ satisfies :
\begin{equation}\label{bcv-eq-tilde-gamma-subsol}
\oplr{\varphi_n}{R_n}(x)+(\tilde\gamma(x)+\lambda_p(\lb{R_n}+\tilde \gamma(x)))\varphi_n(x)= 0\quad \text{ in }\quad B_{R_n}(0).
\end{equation}
Since $J$ is compactly supported and $J(0)>0$ there exist positive constants $\; r_0\ge r_1$ and $ M\ge m\;$ so that 
$$ M\chi_{B_{r_0}(x)}\ge J(x-y)\ge m\chi_{B_{r_1}(x)} \quad \text{for all }\; x,y\in \R^N.$$
For $n$ large enough, say $n\ge n_0$,  we have $R_n>R_1+r_0$ and by the Harnack inequality, for all $n\ge n_0$, we have 
$$\varphi_n(x) \le C(B_{R_1},\lambda_p(\lb{R_n}+\tilde \gamma(x)))\varphi_n(y) \quad \text{ for all }\quad x,y\in B_{R_1}(0),$$
with  $C(B_{R_1},\lambda_p(\lb{R_n}+\tilde \gamma(x)))$ a constant that  only depends  on $\bigcup_{x\in B_{R_1}}B_{r_0}(x)$ and is  monotone decreasing  with respect to  $\inf_{x\in B_{R_n}}(\tilde \gamma(x)+\lambda_p(\lb{R_n}+\tilde \gamma))$.
For all $n\ge n_0$, the function $\tilde \gamma(x)+\lambda_p(\lb{R_n}+\tilde \gamma)$ being uniformly bounded from below by a constant independent of $n$, the constant $C(B_{R_1},\lambda_p(\lb{R_n}+\tilde \gamma))$ is  bounded from above independently of $n$ by a constant $C(B_{R_1})$.  Thus for all $n\ge n_0$;
$$\varphi_{n}(x)\le C(B_{R_1})\varphi_{n}(y) \quad \forall \, x,y \in B_{R_1}.$$ 

In particular, for all $n\ge n_0$ we have :
$$\varphi_{n}(x)\le C(B_{R_1})\varphi_{n}(0)= C(B_{R_1}) \quad \forall \, x \in B_{R_1}.$$

By choosing  $C>C(B_{R_1})$, we  get $$\psi(x)\ge C>C(B_{R_1})\ge \varphi_n(x) \quad  \forall \, x \in B_{R_1}.$$

Set  $w_n:=\psi-\varphi_n$. From \eqref{bcv-eq-tilde-gamma-supersol} and \eqref{bcv-eq-tilde-gamma-subsol} we get 
\begin{align}
&\oplr{w_n}{R_n}(x)+(\tilde\gamma(x)+\lambda_p(\lb{R_n}+\tilde \gamma(x)))w_n(x)\le 0\quad \text{ for }\quad R_1\le |x|<R_n,\label{bcv-eq-tilde-gamma-wn1}\\
&w_n>0 \quad \text{ for }\quad  |x|<R_1.\label{bcv-eq-tilde-gamma-wn2}
\end{align}
By a straightforward application of the Maximum principle, it follows that for all $n\ge n_0$ we have 
$\varphi_n(x)\le \psi$. Indeed, since  $w_n$ is continuous, $w_n$ achieves a minimum at some point $x_0 \in B_{R_n}$. Assume by contradiction that $w_n(x_0)<0$. Then, thanks to \eqref{bcv-eq-tilde-gamma-wn2}, $x_0\in  B_{R_n}\setminus B_{R_1}$ and at this point, by \eqref{bcv-eq-tilde-gamma-wn1} we have the following contradiction

\begin{align*}
0\ge \oplr{w_n}{R_n}(x_0)+(\tilde\gamma(x_0)+\lambda_p(\lb{R_n}+\tilde \gamma(x)))w_n(x_0)&\ge \int_{B_{R_n}}J(x_0-y)w_n(y)\,dy -w_n(x_0)+ \frac{\nu}{4}|w_n(x_0)|, \\
&\ge \int_{B_{R_n}}J(x_0-y)[w_n(y)-w_n(x_0)]\,dy + \frac{\nu}{4}|w_n(x_0)|> 0.
\end{align*}
Hence, for all $n\ge n_0$, we get $\varphi_n\le \psi$ in $B_{R_n}$ which, by sending $n \to \infty$, leads to $\varphi\le \psi$ in $\R^N$. This concludes the proof of the Lemma.
\end{proof}

 
\section{Long time Behaviour}\label{bcv-section-lgt-beha}
In this section, we investigate the long-time behaviour of the positive solution $u(t,x)$ of 
\begin{align}\label{bcv-eq-parab}
&\frac{\partial u}{\partial t}(t,x)=J\star u(t,x) -u(t,x) +f(x,u(t,x)) \quad \text{ in } \quad \R^+\times\R^N,\\ 
&u(0,x)=u_0(x).
\end{align}

For any $u_0\in C^k(\R^N)\cap L^{\infty}$ or in $ C^k(\R^N)\cap L^{1}(\R^N)$ the existence of a smooth  solution $u(t,x)\in C^1((0,+\infty), C^{\min\{1,k\}}(\R^N))$ respectively $u(t,x)\in C^1((0,+\infty), C^{\min\{1,k\}}(\R^N)\cap L^1(\R^N))$   is a straightforward consequence of the Cauchy-Lipschitz Theorem and of the $KPP$ structure of the nonlinearity $f$.
Before going to the proof of the asymptotic behaviour, let us recall some useful results

\begin{lemma}\label{bcv-lem-para-sub-supersol}
Assume that $u_0(x)$ is a sub-solution to \eqref{bcv-eq-parab}, then the solution $u(t,x)$ is increasing in time. Conversely, if  $u_0(x)$ is a super-solution to \eqref{bcv-eq-parab} then $u(t,x)$ is decreasing in time.
\end{lemma}   

The proof of this Lemma follows from a straightforward application of the parabolic maximum principle and is left to reader. Let us now prove the asymptotic behaviour of the solution of \eqref{bcv-eq-parab} and finish the proof of Theorem \ref{bcv-thm1}. 
\begin{proof}  
Let $z(t,x)$ be the solution of
\begin{align}
&\frac{\partial z}{\partial t}=J\star z -z +f(x,z(t,x)) \quad \text{ in } \quad \R^+\times\R^N,\\ 
&z(0,x)=C\|u_0\|_{\infty}.
\end{align}
Since $S(x) \in L^{\infty}$ by choosing $C$ large enough, the constant $C\|u_0\|_{\infty}$ is a super-solution of \eqref{bcv-eq-parab}. Therefore $z(t,x)$ is a decreasing function and by the parabolic maximum principle $u(t,x)\le z(t,x)$ for all $(t,x) \in [0,+\infty)\times \R^N$. Therefore, 

\begin{equation} \label{bcv-eq-parab-1}
 \limsup_{t\to \infty}u(t,x)\le \limsup_{t\to \infty} z(t,x) \quad\text{for all } \quad x \in \R^N.
 \end{equation}
Let us consider the approximated parabolic problem  

 \begin{align}\label{bcv-eq-approx-parab}
&\frac{\partial v_R}{\partial t}(t,x)=\int_{B_R(0)}J(x-y)v_R(t,y)\,dy -v_R(t,x) +f(x,v_R(t,x)) \quad \text{ in } \quad \R^+\times B_R(0),\\ 
&v_R(0,x)=\eta_{_R}u_{0}(x),
\end{align}
where $\eta_{_R}:=\eta\left(\frac{|x|}{R}\right)$ with $\eta \in C(\R^+)$ a smooth cut-off function such that $\eta\ge 0, $ $\eta\equiv 1$ in $[0,1]$ and $\eta \equiv 0$ in $\R^+\setminus[0,2]$.  
 By Theorem \ref{bcv-thm-existence}, for $R$ large enough the solution $v_R$ converges  to $u_{R}$ as $t\to \infty$, where $u_R$ denotes the unique positive stationary solution of  \eqref{bcv-eq-approx-parab}.
 By construction, since $u(t,x)$ is a super-solution of the problem \eqref{bcv-eq-approx-parab}, by the parabolic comparison principle  for  $R$ large enough we have
 $v_R(t,x) \le u(t,x)$ for all $(t,x) \in [0,+\infty)\times B_R(0)$.
 Thus for  $R$ large enough, we get
 \begin{equation} \label{bcv-eq-approx-parab-1}
 \liminf_{t\to \infty}u(t,x)\ge u_R(x) \quad\text{for all } \quad x \in B_R(0).
 \end{equation}
 By taking the limit  in the above inequality as $R \to \infty$, we obtain 
 \begin{equation} \label{bcv-eq-approx-parab-2}
 \liminf_{t\to \infty}u(t,x)\ge \lim_{R\to \infty} u_R(x)=\tilde u(x) \quad\text{for all } \quad x \in \R^N
 \end{equation}
 
Note that we can reproduce the above arguments with $z(t,x)$, and thus we also get 
\begin{align} 
 & v_R(t,x) \le z(t,x) \quad \text{ for all } \quad (t,x) \in [0,+\infty)\times B_R(0),\label{bcv-eq-approx-parab-3}\\
& \liminf_{t\to \infty}z(t,x)\ge \lim_{R\to \infty} u_R(x)=\tilde u(x) \quad\text{for all } \quad x \in \R^N. \label{bcv-eq-approx-parab-4}
 \end{align}
 By \eqref{bcv-eq-approx-parab-3}, $z(t,x)$ is locally uniformly bounded from below and,  since $z(t,x)$ is a decreasing function of $t$, we get $\lim_{t\to \infty}z(t,x)= \bar z(x)>0,$ for all $x\in \R^N$.  Moreover  $\bar z$ is a bounded stationary solution of \eqref{bcv-eq-parab}. By uniqueness of the positive stationary solution, we conclude that $\bar z=\tilde u$.    
 It follows that 
 \begin{equation}
  \lim_{t\to \infty}z(t,x)=\tilde u(x) \quad\text{for all } \quad x \in \R^N \label{bcv-eq-approx-parab-5}
 \end{equation}

By collecting \eqref{bcv-eq-parab-1},\eqref{bcv-eq-approx-parab-2} and \eqref{bcv-eq-approx-parab-5} we get for all $x\in\R^N$

$$ \tilde u(x) \le \liminf_{t\to \infty}u(t,x)   \le \limsup_{t\to \infty}u(t,x) \le \limsup_{t\to \infty}z(t,x)= \lim_{t\to \infty} z(t,x)=\tilde u(x).$$

Now, to complete the proof it remains to show that $\|u-\tilde u\|_{\infty}\to 0$ as $t\to \infty$. To this end, we follow the argument in \cite{Berestycki2008}. 
We argue by contradiction and assume that there exists $\varepsilon>0$ and  sequences $(t_n)\in\R^+$, $(x_n)\in\R^N$
such that
\begin{equation}
\lim_{n\to\infty}t_n=\infty,\quad\quad|{u}(t_n,x_n)-\tilde{u}(x_n)|>\varepsilon,\quad\quad\forall n\in\N.\label{bcv-eq-approx-parab-6}
\end{equation}
By \eqref{bcv-eq-approx-parab-5}, we already know that $u\to \tilde u$ locally uniformly in $\R^N$, so, without loss of generality, we can assume that $|x_n|\to\infty$. 
From the construction of $\tilde u$,  Subsection \ref{bcv-ss-existence}, we have $\lim_{|x|\to \infty}\tilde u(x)= 0$. Therefore, for some $R_0>0$, we have $\tilde u(x)\le \frac{\eps}{2}$ for all $|x|\ge R_0$. This, combined with \eqref{bcv-eq-approx-parab-5} and \eqref{bcv-eq-approx-parab-6}  enforces
\begin{equation}
z(t_n,x_n)-\tilde u(x_n)\ge u(t_n,x_n)-\tilde u(x_n)>\eps,\quad\quad\forall n\in\N.\label{bcv-eq-approx-parab-7}
\end{equation}

Next we require the following limiting result
\begin{lemma}
For all sequences $(t_n)_{n\in \N},(x_n)_{n\in \N}$ such that $\lim_{n\to \infty }t_n =\lim_{n\to\infty}|x_n|= +\infty$, we have 
$z(t_n,x_n)\to 0$.
 \end{lemma}
Assume for the moment that the Lemma holds. Then we obtain a straightforward  contradiction since :
$$
0=\lim_{n\to \infty} z(t_n,x_n)-\tilde u(x_n) \ge \lim_{n\to\infty }u(t_n,x_n)-\tilde u(x_n)>\eps.$$

We now prove the Lemma. Again, we argue by contradiction and assume that there exists $\eps>0$ and sequences $(t_n)_{n\in\N},(x_n)_{n\in \N}$ satisfying $\lim_{n\to \infty}t_n=\lim_{n\to \infty}|x_n|=\infty$ such that $z(t_n,x_n)>\eps$ for all $n \in \N$.
  Let us define $z_n(t,x):=z(t,x+x_n)$. It satisfies
\begin{align*}
&\frac{\partial z_n}{\partial t}(t,x)=\int_{\R}J(x-y)z_n(y)\,dy -z_n(t,x) +f(x+x_n,z_n(t,x)) \quad \text{ in } \quad \R^+\times \R^N,\\ 
&z_n(0,x)=C\|u_0\|_{\infty},
\end{align*}
  and $ 0<z_n(t,x)<C\|u_0\|_{\infty}$ for $t>0$. Since for all $n$, $z_n(0,x)\in C^{\infty},$ by the Cauchy Lipschitz Theorem, we see that  $z_n\in C^{1}(\R^+,C^1(\R^N))$. Thus, there exists $C_0>0$ independent of $n$ so that  $\|z_n\|_{C^{1,1}(\R^+, C(\R^N))} <C_0$.  
From these estimates, the sequence  $(z_n)_{n\in \N}$ is  uniformly bounded in $C^{1,1}((0,T),C^{0,1}(\R^N))$ for any $T>0$.
By a diagonal extraction,  there exists a subsequence of $(z_n)_{n\in\N}$ that converges  locally uniformly to $\tilde z(t,x)$. Moreover, thanks to $\lim_{|x|\to \infty}\frac{f(x,s)}{s}<0$, there exists $\kappa>0$ so that $\tilde z(x,t)$ satisfies
 \begin{align}\label{bcv-eq-approx-parab-8}
&\frac{\partial \tilde z}{\partial t}(t,x)\le \int_{\R}J(x-y)\tilde z(t,y)\,dy -\tilde z(t,x) -\kappa \tilde z(t,x)) \quad \text{ in } \quad \R^+\times \R^N,\\ 
&\tilde z(0,x)=C\|u_0\|_{\infty}.
\end{align}

In addition, for all $t>0$, $\tilde z(t,0)=\lim_{n\to\infty}z_n(t,0)\ge \eps$. Since $\tilde z(0,x)$ is a super-solution of \eqref{bcv-eq-approx-parab-8}, by Lemma \ref{bcv-lem-para-sub-supersol} the function $\tilde z(t,x)$ is monotone decreasing in time. By sending $t\to \infty$, since $\tilde z\ge 0$, $\tilde z$ converges locally uniformly to a non-negative  function $\bar z$ that satisfies
\begin{align*}\label{bcv-eq-approx-parab-9}
&\int_{\R}J(x-y)\bar z(y)\,dy -\bar z(x) -\kappa \bar z(x))\ge 0 \quad \text{ in } \quad  \R^N,\\
&0\le \bar z\le C\|u_0\|_{\infty},\\
&\bar z(0)\ge \eps.
\end{align*}
Let us now consider the function $w(x):=\frac{\eps}{2}e^{\alpha|x|}-\bar z$ with $\alpha$ to be chosen.  It satisfies
$$
\int_{\R}J(x-y)w(y)\,dy -w(x) -\kappa w (x)\le  \rho e^{\alpha|x|}\left(\int_{\R^N}J(z)e^{\alpha|z|}\,dy -1-\kappa\right) \quad \text{ in } \quad  \R^N.
$$
The left hand side of the inequality is well defined and continuous with respect to $\alpha$ since $J$ is compactly supported. 
Since  $\int_{\R^N}J(z)dz=1$, by  choosing $\alpha$ small enough  we achieve
$$
\int_{\R}J(x-y)w(y)\,dy -w(x) -\kappa w (x)< 0 \quad \text{ in } \quad  \R^N.
$$ 
By construction, since $\bar z$ is bounded,  $\lim_{|x|\to \infty}w(x)=+\infty$ and  $w$ achieves a minimum in $\R^N$, say at $x_0$.
Since $w(0)=\frac{\eps}{2}-\bar z(0)\le -\frac{\eps}{2}$, we have $w(x_0)<0$.  At this point, we get the following contradiction 
  $$
0<\int_{\R}J(x_0-y)[w(y)-w(x_0)]\,dy  -\kappa w (x_0)< 0 \quad \text{ in } \quad  \R^N.
$$

 \end{proof}

 Finally we establish the long time behaviour  of the solution  $u(t,x)$ starting from an integrable initial datum $u_0$,i.e   $u_0\in L^1(\R^N)\cap C(\R^N)$. To do so, we consider two auxiliary  functions $h(t,x)$ and $v(t,x)$ that are respectively solutions of 
 \begin{equation}\label{bcv-eq-parab-h}
 \begin{cases}
\frac{\partial h}{\partial t}(t,x)=J\star h(t,x) -h(t,x) +f(x,h(t,x)) \quad \text{ in } \quad \R^+\times\R^N,\\ 
h(0,x)=\sup\{\tilde u(x),u_0(x)\},
\end{cases}
\end{equation}

 \begin{equation}\label{bcv-eq-parab-v}
 \begin{cases}
\frac{\partial v}{\partial t}(t,x)=J\star v(t,x) -v(t,x) +f(x,v(t,x)) \quad \text{ in } \quad \R^+\times\R^N,\\ 
v(0,x)=\inf\{\tilde u(x),u_0(x)\}.
\end{cases}
\end{equation}

 By construction, from the comparison principle  we see that $v(t,x)\le u(t,x)\le h(t,x)$ for all $(t,x)\in \R^+\times \R^N$. Therefore 
$$\|u-\tilde u\|_{L^1(\R^N)}\le \sup\{ \|h-\tilde u\|_{L^1(\R^N)},\|v-\tilde u\|_{L^1(\R^N)}\}.$$ 
Thus,  to prove that $\|u-\tilde u\|_{L^{1}(\R^N)}\to 0$ it is enough to show that $h$ and $v$ converge to $\tilde u$ in $L^1(\R^N)$.
 
Let us show that  $v$ converges to $\tilde u$ in $L^1(\R^N)$. Since $\tilde u(x)$ is a super solution to \eqref{bcv-eq-parab-v}   we deduce $v(t,x)\le \tilde u(x)$ for all $x\in \R^N$.
Let $\eps>0$ be fixed  and  choose  $R$ such that $\int_{\R^N\setminus B(0,R)}\tilde u(x)\, dx \le \frac{\eps}{4}$.  
We have 

\begin{align*}
\|\tilde u-v\|_{L^1(\R^N)}&=\int_{\R^N\setminus B_R(0)}(\tilde u(x)-v(t,x))\,dx+\int_{B_R(0)}(\tilde u(x)-v(t,x))\,dx,\\
&\le 2\int_{\R^N\setminus B_R(0)}\tilde u(x)\,dx+\int_{B_R(0)}(\tilde u(x)-v(t,x))\,dx,\\
&\le \frac{\eps}{2}+\int_{B_R(0)}(\tilde u(x)-v(t,x))\,dx.
\end{align*}
Recall that  $v$ converges pointwise to $\tilde u$ as $t$ tends to infinity. Therefore, by Lebesgue's Theorem for some $t(\eps)$ we get for  $t\ge t(\eps)$, $\int_{B_R(0)}(\tilde u(x)-v(t,x))\,dx\le \frac{\eps}{2}$ which yields
  
$$ \|\tilde u-v\|_{L^1(\R^N)}\le \eps.$$
 Since $\eps$ is arbitrary, we see that  $\lim_{t\to \infty} \|\tilde u-v\|_{L^1(\R^N)}=0$.

 To obtain that  $\|h-\tilde u\|_{L^{1}(\R^N)} \to 0$ we argue as follows. By construction, $\tilde u $ is a sub solution to \eqref{bcv-eq-parab-h}, thus, $\tilde u(x) \le h(t,x)$ for  $(t,x)\in \R^+\times \R^N$. Let us denote $w(t,x):=h(t,x)-\tilde u(x)$. Then $w$ satisfies for  $(t,x)\in \R^+\times\R^N$:
 \begin{align*}
 \frac{\partial w}{\partial t}(t,x)&= J\star w(t,x) -w(t,x) +\left(\frac{f(x,h(t,x))}{h(t,x)}-\frac{f(x,\tilde u(x))}{\tilde u}\right)h(t,x) + \frac{f(x,\tilde u(x))}{\tilde u}w(t,x), \\
 &\le J\star w(t,x) -w(t,x) + \frac{f(x,\tilde u(x))}{\tilde u}w(t,x).
    \end{align*}
Now since  $\lim_{|x|\to \infty}\frac{f(x,s)}{s}<0$, there exists $\kappa>0$ and $R_0$ so that $w$ satisfies
\begin{equation}\label{bcv-eq-parab-w}
\frac{\partial w}{\partial t}(t,x)\le J\star w(t,x) -w(t,x) -\kappa w(t,x) \quad \text{ in } \quad \R^+\times\R^N\setminus B_{R_0}(0).
\end{equation}

Fix now $\eps>0$. Recall that $h(t,x)$ converges pointwise to $ \tilde u$. By Lebesgue's Theorem, there exists $t_0$ so that for $t\ge t_0,$ 
$$\int_{B_{R_0}(0)}w(t,x)\,dx \le \kappa \eps.$$ 

Now let us estimate $\int_{\R^N\setminus B_{R_0}(0)}w(x)\, dx $ for $t\ge t_0$. Integrating \eqref{bcv-eq-parab-w} over $\R^N \setminus B_{R_0}(0)$ yields
\begin{align*}
 \frac{\partial \int_{\R^N\setminus B_{R_0}(0)}w(t,x)\,dx}{\partial t}&\le \int_{\R^N\setminus B_{R_0}(0)}J\star w(t,x)\,dx -\int_{\R^N\setminus B_{R_0}(0)}w(t,x)\,dx -\kappa\int_{\R^N\setminus B_{R_0}(0)}w(t,x)\,dx.
    \end{align*}
  By  Fubini's Theorem, the uniform estimate on $\|w\|_{\infty}$ and the unit mass of the kernel, we can check that for $t\ge t_0$ 
  \begin{align*}
  \int_{\R^N\setminus B_{R_0}(0)}J\star w(t,x)\,dx
  & \le \int_{\R^N\setminus B_{R_0}(0)}w(t,y)\,dy+\int_{B_{R_0}(0)}w(t,y)\,dy,\\
  & \le \int_{\R^N\setminus B_{R_0}(0)}w(t,y)\,dy+\kappa \eps.
  \end{align*} 
  Therefore for $t\ge t_0$, $w$ satisfies 
  \begin{equation*}
 \frac{\partial \int_{\R^N\setminus B_{R_0}(0)}w(t,x)\,dx}{\partial t}\le \kappa \eps -\kappa\int_{\R^N\setminus B_{R_0}(0)}w(t,x)\,dx.
    \end{equation*} 
 From this differential inequality, there exists $t(\eps) \ge t_0$ such that for all  $t\ge t(\eps) $ we have 
 $$\int_{\R^N\setminus B_{R_0}(0)}w(t,x)\,dx \le 2\eps.$$
 Hence, for  $t\ge t(\eps)$ we have
 $$ \|w\|_{L^1(\R^N)}= \int_{\R^N\setminus B_{R_0}(0)}w(t,x)\,dx +\int_{B_{R_0}(0)}w(t,x)\,dx \le \left(2 +\frac{\kappa}{|B_{R_0}(0)|}\right)\eps, $$
As above, $\eps$ being arbitrary,  we see that  $\lim_{t\to \infty} \|w\|_{L^1(\R^N)}=0$.
 
 \section{Some asymptotics}\label{bcv-section-asym}
 
 In this section we analyse the qualitative behaviour of the solution of \eqref{bcv-eq} with respect to the size of the support of $J$.
 For convenience we investigate the particular situation 
\begin{equation} \label{bcv-eq-asb}\tag{$P_\eps$}
\frac{1}{\eps^m}(J_\eps\star u -u) +u(a(x)-u)=0\quad \text{ in  }\quad \R^N 
 \end{equation}
 where $J_\eps(z)=\frac{1}{\eps^N}J\left(\frac{z}{\eps}\right)$ with $\textrm{supp}(J)=\bar B(0,1)$  and $a \in C^1(\R^N)$  so that $a^+\not\equiv 0$.
This last condition on $a(\cdot)$  is  necessary for the existence of a solution. Indeed, if  $a^+\equiv 0$ then for any positive constant $c_0$ we have  $$ \opm{c_0}+a(x)c_0\le 0.$$ 
Therefore, $\lambda_p(\opm{c_0}+a(x)c_0) \ge 0$ and for all $\eps$ there is no solution of \eqref{bcv-eq-asb} besides $0$.

We analyse the behaviour of $u_\eps$ when $\eps \to 0$ and $\to +\infty$ and seek to understand the influence of $m$ on the resulting limits. Now we start by showing some \textit{a priori} estimates for the solution $u_\eps$.

\begin{lemma}\label{bcv-lem-esti1} There exists  positive constants $C_1,C_2,C_3$ such that   for any positive bounded solution   $u_\eps$ of \eqref{bcv-eq-asb} the following estimates hold
\begin{itemize}
\item[(i)] $\nlp{u_\eps}{2}{\R^N}\le C_1, \quad \|u_\eps\|_{\infty}<C_3$,
\item[(ii)] $ \int_{\R^N}\int_{\R^N}J_\eps(x-y)(u_\eps(x)-u_{\eps}(y))^2\,dxdy\le C_2\eps^m$ 
\item[(iii)] 
$\sup_{supp(a^+)}u_\eps  \ge -\frac{\lambda_p(\m_{\eps,m}+a(x))}{2}.$
\item[(iv)] $u_\eps\ge (a(x)-\frac{1}{\eps^m})^+$,  
\end{itemize} 
 \end{lemma}

\begin{proof}
By construction the solution is unique and $u_\eps\in L^1(\R^N)\cap L^{\infty}$. By \eqref{bcv-eq-asb}, $u_\eps\le M=\|a\|_{\infty}$. We derive $(i)$ by integrating \eqref{bcv-eq-asb} over $\R^N$. Indeed,  we get
$$\int_{\R^N}u_\eps^2(x)\,dx=\int_{\R^N}a(x)u_\eps(x)\,dx\le \int_{\R^N}a^+(x)u_\eps(x)\,dx\le M\int_{\R^N}a^+(x)dx=:C_1.$$
To obtain $(ii)$, let us multiply \eqref{bcv-eq-asb} by $u_\eps$ and integrate  over $\R^N$ to get
$$\frac{1}{2\eps^m}\int_{\R^N}\int_{\R^N}J(x-y)(u_\eps(x)-u_\eps(y))^2\,dxdy=\int_{\R^N}u_\eps^2(x)(a(x)-u_{\eps}(x))\,dx. $$
Since $u_\eps$ and $a(x)$ are uniformly bounded independently of $\eps$, $(ii)$ holds  with $C_2:=4C_1M$.
 Observe  that $(a(x)-\frac{1}{\eps^m})^+$ is always a sub-solution of \eqref{bcv-eq-asb}, so by the standard sweeping principle 
 $u_\eps\ge(a(x)-\frac{1}{\eps^m})^+ $ and $(iv)$ holds.

Finally let us derive $(iii)$. Since $u_\eps$ is a positive bounded solution of \eqref{bcv-eq-asb} by Theorem \ref{bcv-thm1} we know that $\lambda_p(\m_{\eps,m}+a)<0$. Consequently, and as  $J$ is compactly supported, by regularising $a$ if need be, we can find   $\varphi_\eps\in C_c(\R^N)$ so that 
$$\m_{\eps,m}[\varphi_\eps](x)+(a(x)+\frac{\lambda_p}{2})\varphi_\eps(x)\ge 0 \quad \text{ in }\quad \R^N,$$
see the proof of Lemma 3.1 in \cite{Berestycki2014}.
Moreover, we can normalised $\varphi_\eps$ by imposing $\|\varphi_\eps\|_{\infty}=1$.
Plugging $\theta\varphi_\eps$  in \eqref{bcv-eq-asb}, it follows that 

$$\m_{\eps,m}[\theta\varphi_\eps]+\theta\varphi_\eps(a(x)-\theta\varphi_\eps)\ge \theta\varphi_\eps\left(-\frac{\lambda_p}{2}-\theta\varphi_\eps\right).$$
 Therefore,  for $0< \theta\le -\frac{\lambda_p}{2}$, the function $\theta\varphi_\eps$ is a sub-solution to \eqref{bcv-eq-asb}. By the sweeping argument we have already mentioned, we get 
 $$-\frac{\lambda_p}{2}\varphi_\eps\le u_\eps\quad \text{  and }\quad \sup_{\R^N}u_\eps \ge -\frac{\lambda_p}{2}.  $$ 
     
Since $u_\eps\in L^1(\R^N)$, $u_\eps$ achieves its maximum at some point, say $x_0$. From \eqref{bcv-eq-asb}, we infer 

$$0\ge \m_{\eps,m}[u_\eps](x_0)=-u_\eps(x_0)(a(x_0)-u_\eps(x_0)).$$
This implies that $x_0\in \textrm{supp}(a^+)$ and thus $\|u_\eps\|_{\infty}=  \sup_{\textrm{supp}(a^+)} u_\eps$ which proves $(iii)$. 
\end{proof}

Next we derive an upper bound  for large $\eps$.

\begin{lemma}\label{bcv-lem-esti2} There exists $\eps_0>0$  such that for all $m\ge 0$ and $\eps \ge \eps_0, $ any positive bounded solution  $u_\eps$ of \eqref{bcv-eq-asb} satisfies $$ u_\eps\le a^+(x)+\frac{1}{\eps^{\frac{N}{4}}}  $$
 \end{lemma}
\begin{proof}
Let $\delta \in (0,\frac{N}{2})$ and consider the function $\zeta_\eps(x):=\frac{1}{\eps^{\frac{N}{2}-\delta}}+a^+(x)$. We will show that $\zeta_\eps$ is a super-solution to \eqref{bcv-eq-asb} when $\eps$ is large enough.
 
 Indeed, we have 
 \begin{align*}
 \opmem{\zeta_\eps}{m}(x)+\zeta_\eps(x)(a(x)-\zeta_\eps(x)) 
 &\le\frac{\|J\|_{\infty}}{\eps^{N+m}}\int_{\R^N}a^+(y)\,dy+\left(\frac{1}{\eps^{\frac{N}{2}-\delta}}+a^+(x)\right)\left[a(x)-\frac{1}{\eps^{\frac{N}{2}-\delta}}-a^+(x)\right]\\
 &\le\frac{\|J\|_{\infty}}{\eps^{N+m}}\int_{\R^N}a^+(y)\,dy-\frac{1}{\eps^{N-2\delta}}.
 \end{align*} 
 where in the last inequality we use  
 $$\left(\frac{1}{\eps^{\frac{N}{2}-\delta}}+a^+(x)\right)\left[a(x)-\frac{1}{\eps^{\frac{N}{2}-\delta}}-a^+(x)\right]\le - \frac{1}{\eps^{N-2\delta}}\quad \text{for all } \quad x\in \R^N.$$
 Thus, for $\eps$ large enough, we get
  \begin{align*}
 \opm{\zeta_\eps}(x)+\zeta_\eps(x)(a(x)-\zeta_\eps(x))\le\frac{\|J\|_{\infty}}{\eps^{N+m}}\int_{\R^N}a^+(y)\,dy-\frac{1}{\eps^{N-2\delta}}<0.
 \end{align*}
  Therefore for $\eps>>1$,  we get $u_\eps\le \zeta_\eps$.  We end the proof by taking $\delta =\frac{N}{4}$.
      
 \end{proof} 
 
\begin{remark}\label{bcv-rem-m0}
When $m=0$ and $(a(x)-1)^+\not \equiv 0$, the above computation holds as well with $\zeta_\eps(x):=\frac{1}{\eps^{\frac{N}{2}-\delta}}+(a(x)-1)^+$. Thus in this case, for large $\eps$, we have
$$u_\eps(x)\le \frac{1}{\eps^{\frac{N}{4}}}+(a(x)-1)^+.$$
\end{remark} 
 
 Next,  we prove the continuity of  $\lambda_p(\lb{R,\eps} +a(x))$ with respect to $\eps$.
\begin{lemma}\label{bcv-lem-esti-conti}
Let $R,\eps$ be fixed and positive  then for all $\eta>0$ there exists $\delta>0$ so that $$|\lambda_p(\lb{R} +a_{\eps}(x))-\lambda_p(\lb{R} +a_{\eps+\delta}(x))|\le \eta,$$
where $a_\eps(x):=a(\eps x)$.
\end{lemma}

\begin{proof}
Let $\eps>0$ and $R>0$ be fixed. We observe that for all $|\delta|<\eps$ we have for all $x\in \R^N, a_{\eps+\delta}(x)= a_\eps\left(\frac{\eps+\delta}{\eps}x\right)$ therefore 
$$\|a_\eps -a_{\eps+\delta}\|_{\infty,R}=\sup_{B(0,R)}\left\|a_\eps(x) -a_\eps\left(\frac{\eps+\delta}{\eps}x\right)\right\|. $$
Since $a_\eps$ is  Lipschitz continuous in $\R^N$, we have 
$$\left\|a_\eps(x) -a_\eps\left(\frac{\eps+\delta}{\eps}x\right)\right\|\le K(\eps)\eps \delta\|x\|,$$
where $K(\eps)$ is the Lipschitz constant of $a_\eps$.
Thus  
$$\|a_\eps -a_{\eps+\delta}\|_{\infty,R}\le K(\eps)R\eps\delta.$$
Hence, by  Proposition \ref{bcv-prop-pev} (ii) we get 
$$|\lambda_p(\lb{R}+a_\eps(x)) -\lambda_p(\lb{R}+a_{\eps+\delta}(x))|\le K(\eps)R\eps\delta.$$
\end{proof}

Lastly, we require the following identity.
\begin{proposition} Let $\rho \in C_c^{\infty}(\R^N)$ be a radial function, then for all $u\in L^2(\R^N), \varphi \in C^{\infty}_c(\R^N)$  we have  
$$\iint_{\R^N\times \R^N}\rho(z)[u(x+z)-u(x)]\varphi(x)\,dzdx= \frac{1}{2}\iint_{\R^N\times \R^N}\rho(z)u(x)[\varphi(x+z)-2\varphi(x)+\varphi(x-z)]\,dzdx.$$
\end{proposition}

\begin{proof}
Let us denote the left hand side of the equality by $I$, i.e.  
$$I:= \iint_{\R^N\times \R^N}\rho(z)[u(x+z)-u(x)]\varphi(x)\,dzdx.$$
By change of variables and  thanks to the spherical  symmetry of $\rho$,  we get
 \begin{align*}
 I&=\frac{1}{2} \iint_{\R^N\times \R^N}\rho(z)[u(x+z)-u(x)]\varphi(x) +\frac{1}{2} \iint_{\R^N\times \R^N}\rho(-z)[u(x-z)-u(x)]\varphi(x),\\
 &=\frac{1}{2} \iint_{\R^N\times \R^N}\rho(z)[u(x+z)-u(x)]\varphi(x) +\frac{1}{2} \iint_{\R^N\times \R^N}\rho(z)[u(x)-u(x+z)]\varphi(x+z),\\
 &=-\frac{1}{2} \iint_{\R^N\times \R^N}\rho(z)[u(x+z)-u(x)][\varphi(x+z)-\varphi(x)],\\
 &=-\frac{1}{2} \iint_{\R^N\times \R^N}\rho(z)u(x)[\varphi(x)-\varphi(x-z)] +\frac{1}{2} \iint_{\R^N\times \R^N}\rho(z)u(x)[\varphi(x+z)-\varphi(x)],\\
 &=\frac{1}{2} \iint_{\R^N\times \R^N}\rho(z)u(x)[\varphi(x+z)-2\varphi(x)+\varphi(x-z)].
 \end{align*}
\end{proof}

From this Proposition,   we get the following identity for all $u\in L^2(\R^N), \varphi \in C^{\infty}_c(\R^N)$ :
\begin{equation}\label{bcv-eq-asb-id}
\int_{\R^N}\opmem{u}{m}(x)\varphi(x)\,dx= \frac{\eps^{2-m}D_2(J)}{2}\iint_{\R^N\times\R^N}\frac{\rho_\eps(z)}{|z|^2}u_\eps(x)[\varphi(x+z)-2\varphi(x)+\varphi(x-z)]\,dxdz
\end{equation} 
where $\rho_\eps(z)=\frac{1}{\eps^ND_2(J)}J\left(\frac{z}{\eps}\right)\frac{|z|^2}{\eps^2}$.

With these various \textit{apriori} estimates, we can now analyse the asymptotic behaviour of $u_\eps$.

\subsection{The case $m=0$}
In this situation, from Theorem \ref{bcv-thm6} we know that 
\begin{align}
&\lim_{\eps\to 0}\lambda_p(\M_\eps +a(x))= -\sup_{\R^N}a(x)\label{bcv-eq-asl-m0-1}\\
&\lim_{\eps\to +\infty}\lambda_p(\M_\eps +a(x))= 1-\sup_{\R^N}a(x)\label{bcv-eq-asl-m0-2}
\end{align}

As a consequence, for $\eps$ small enough,  $\lambda_p(\M_\eps +a(x))\le -\frac{\sup_{\R^N}a(x)}{2}<0$ and, 
 by Theorem \eqref{bcv-thm1}, there exists a positive solution of \eqref{bcv-eq-asb}.
Moreover the following quantity is well defined 
 $$\eps^*:=\sup\{\eps>0\, |\, \text{ for all }\, \eps'<\eps,\;\text{ there exists a positive solution to } (P_{\eps'}) \}.$$ 
  In view of \eqref{bcv-eq-asl-m0-2}  $\eps^*\in (0,+\infty]$ and $\eps^*<+\infty$ if and only if $(a(x)-1)^+\not\equiv 0$.
\medskip

Let us now determine the limit of $u_\eps$ as $\eps \to 0$ and $\eps \to+\infty$.  We start by proving that  

\begin{equation} \label{bcv-eq-lim0-m0}
\lim_{\eps \to 0}u_\eps(x)= v(x) \quad\text{a.e.}
\end{equation}
where $v$ is a non negative bounded solution of 
\begin{equation}
v(x)(a(x)-v(x))=0 \quad \text{in}\quad \R^N.\label{bcv-eq-eqlim}
\end{equation}
Let $w_\eps:=a(x)-u_\eps$, then from \eqref{bcv-eq-asb}, $w_\eps$ satisfy
\begin{equation}\label{bcv-eq-asb2}
-J_\eps\star w_\eps +w_\eps +u_\eps(x)w_\eps(x)=a(x)-J_\eps\star a(x).
\end{equation}
Multiplying this equation by $w_\eps^+$ and integrating over $\R^N$, it follows that 
$$\iint_{\R^N\times \R^N}J_\eps(x-y)((w_\eps^+)^2(x) -w_\eps(y)w_\eps^+(x))\,dxdy +\int_{\R^N}u_\eps (w_\eps^+)^2=\int_{\R^N}w^+_\eps g_\eps,$$ 
with $g_\eps(x):= a(x)-J_\eps\star a(x)$.

Let us now estimate the above integrals. 
First, we observe that the double integral is positive. Indeed, since $w(y)=w^+(y)-w^-(y)$ we get 

\begin{equation*}
\begin{split}
\iint_{\R^N\times \R^N}J_\eps(x-y)((w_\eps^+)^2(x) -w_\eps(y)w_\eps^+(x))\,dxdy=\iint_{\R^N\times \R^N}J_\eps(x-y)((w_\eps^+)^2(x) -w^+_\eps(y)w_\eps^+(x))\,dxdy \\+\iint_{\R^N\times \R^N}J_\eps(x-y)w^-_\eps(y)w_\eps^+(x)\,dxdy.
\end{split}
\end{equation*}
Thus,
\begin{equation}\label{bcv-eq-asb-m0-epsto0-1}
\begin{split}
\iint_{\R^N\times \R^N}J_\eps(x-y)((w_\eps^+)^2(x) -w_\eps(y)w_\eps^+(x))\,dxdy=\frac{1}{2} \iint_{\R^N\times \R^N}J_\eps(x-y)((w_\eps^+)(x) -w^+_\eps(y))^2\,dxdy \\+\iint_{\R^N\times \R^N}J_\eps(x-y)w^-_\eps(y)w_\eps^+(x)\,dxdy.
\end{split}
\end{equation}
Let us denote $Q:=\textrm{supp}(a^+)$. Since $u_\eps$ is positive and uniformly bounded, we have $\textrm{supp}(w^+)\subset Q$ and  
$$\left| \int_{\R^N}w^+_\eps g_\eps\right| \le C\int_{Q}|g_\eps|.$$
Since $a$ is Lipschitz continuous, a Taylor expansion leads to  $|g_\eps(x)|\le \eps D_2(J)\|\nabla a\|_{\infty}$.
Therefore, 
 \begin{equation}\label{bcv-eq-asb-m0-epsto0-2}
 \left| \int_{\R^N}w^+_\eps g_\eps\right| \le C|Q|\eps.
 \end{equation}
 
Collecting \eqref{bcv-eq-asb-m0-epsto0-1} and \eqref{bcv-eq-asb-m0-epsto0-2}, we get

$$\frac{1}{2} \iint_{\R^N\times \R^N}J_\eps(x-y)((w_\eps^+)(x) -w^+_\eps(y))^2\,dxdy +\iint_{\R^N\times \R^N}J_\eps(x-y)w^-_\eps(y)w_\eps^+(x)\,dxdy +\int_{\R^N}u_\eps (w_\eps^+)^2\le C\eps.$$ 
    
 Thus, $$\int_{\R^N}u_\eps (w_\eps^+)^2\le C\eps $$ and 
 $ u_\eps w_\eps^+ \to 0$ almost everywhere in $Q$. 
 
 Recalling  that $$\int_{\R^N}u_\eps w_\eps=0,$$  from the above estimates we conclude that 
 
 $$\int_{\R^N\setminus Q}u_\eps(a(x)-u_\eps)=\int_{Q}u_\eps(a(x)-u_\eps)\to 0 \quad \text{when }\quad \eps \to 0.$$
 
Since $u_\eps(a(x)-u_\eps)\le 0$ in $\R^N\setminus Q$,  it follows that 
$ u_\eps(x) w_\eps \to 0$ almost everywhere in $\R^N\setminus Q$. Since $u_\eps>0$ and $w_\eps=(a(x)-u_\eps(x))\le 0$, it follows that 
$u_\eps\to 0$ almost everywhere in $\R^N\setminus Q$. 
 Hence, $u_\eps$ converges pointwise almost everywhere  to a bounded non-neqative solution of \eqref{bcv-eq-eqlim}.

\begin{remark}\label{bcv-rem-proof-conv}
Note that the above proof can easily be adapted to $\m_{\eps,m}$ for  $m<2$  as soon as the function $a$ is smooth enough. Indeed, for $a\in C^2(\R^N)$, following the above arguments, we get by using the Taylor expansion up to order $2$ of $a$
$$\int_{\R^N}u_\eps (w_\eps^+)^2(x)\,dx\le C\eps^{2-m}, $$
with a constant $C$ only depending on $\|\nabla^2u\|_{\infty}$.
When $a$ is only Lipschitz, this argument is  valid for $\m_{\eps,m}$ only when  $m<1$. 
\end{remark}

Finally, to complete our analysis,  we need to check that  
\begin{equation}
\lim_{\eps\to\eps^*}u_\eps =(a(x)-1)^+. \label{bcv-eq-liminfi-m0}
\end{equation}

We treat separately the following two cases : $(i)\; \eps^*<+\infty,\; (ii)\; \eps^*=\infty.$ The latter case arises when $\sup_{\R^N}(a(x)-1)>0$. In this situation, there exists $R_0>0$ such that  the continuous function  $\varphi=\left(a(x)-1\right)^+\not\equiv 0$ in $B_{R}(0)$ for $R\ge R_0$ and we can check that $\varphi$ is a sub-solution for the approximated problem: 
 \begin{equation}
\int_{B_R(0)}J_\eps(x-y) u(y)\,dy -u(x) +u(x)(a(x)-u)=0\quad \text{ in  }\quad B_R(0). \label{bcv-eq-asb-approx}
 \end{equation}
 Since large constants are super-solutions of \eqref{bcv-eq-asb-approx} for any $\eps\ge 0, R>R_0$, there exists a unique solution $u_{\eps, R}$ with $\varphi\le u_{\eps,R}\le M$. By sending $R \to \infty$ and by the uniqueness of the solution of \eqref{bcv-eq-asb} we have 
 $\varphi \le u_{\eps} \le M$ in $\R^N$.

 \subsubsection*{Case $\eps^*=+\infty$:}
 
 Owing to Lemma \ref{bcv-lem-esti1} and by Remark \ref{bcv-rem-m0}, for all $x \in \R^N$ for large $\eps$ we get 
$$(a(x)-1)^+\le u_\eps(x)\le (a(x)-1)^++\frac{1}{\eps^{\frac{N}{4}}}.$$ 
 Hence, $u_\eps$ converge uniformly to $(a(x)-1)^+$. 

  \subsubsection*{Case $\eps^*<+\infty$:}

  In this situation, the function $(a(x)-1)^+\equiv 0$ in $\R^N$ and the problem is  reduced to prove that 
$$\lim_{\eps\to\eps^*}u_\eps(x)=0 \quad \text{ for all }\quad x\in \R^N.$$
Note that by definition of $\eps^*$ we must have 
$\lambda_p(\M_{\eps^*} +a(x))\ge 0$. 
Indeed, if not, then $\lambda_p(\M_{\eps^*} +a(x))< 0$ and by Lemma \ref{bcv-lem-scal-eq},  $\lambda_p(\M +a_{\eps^*}(x))< 0$. This implies that  for some $R$,   $\lambda_p(\lb{R} +a_{\eps^*}(x))< 0$.
By continuity of $\lambda_p(\lb{R} +a_{\eps^*}(x))$ with respect to $\eps$, (Lemma \ref{bcv-lem-esti-conti}) we get for some $\delta_0>0$, 
$\lambda_p(\lb{R} +a_{\eps^*+\delta}(x))< 0$ for any $\delta \le \delta_0$. Hence, $\lambda_p(\m_{\eps^*+\delta} +a(x))=\lambda_p(\m +a_{\eps^*+\delta}(x))< 0 $ for any $\delta\le \delta_0$ and  by Theorem \ref{bcv-thm1} there exists a positive solution of \eqref{bcv-eq-asb} for all $\eps\le \eps^*+\delta_0$ thus contradicting the definition of $\eps^*$. 

Note also that since $\eps^*<+\infty$, the construction of the supersolution  in Section \ref{bcv-section-crit} holds for any $\eps \in [\frac{\eps^*}{2},\eps^*]$, thus $u_\eps$ is uniformly bounded in $L^1(\R^N)$.   

 Let $g(x,s):=s(a(x)-1 -s)$ then for all $\eps$ we have 
$$J_\eps\star u_\eps=-g(x,u_\eps(x)) \quad \text{ in }\quad \R^N.$$
Since $J$ is $C^1$ and $\eps^*>0,$ for $\eps \in [\frac{1}{2}\eps^*,\eps^*)$, we have 
\begin{align*} 
|g(x,u_\eps(x)) -g(z,u_\eps(z))|&= \left|\int_{\R^N} [J_\eps(x-y)-J(z-y)]u_\eps(y)\,dy\right|,\\
&\le |x-z| \int_{\R^N} \frac{|J_\eps(x-y)-J(z-y)|}{|x-z|}u_\eps(y)\,dy,\\
&\le C(\eps^*)|x-z|.
  \end{align*}
This leads to : 
\begin{align*}
C(\eps^*)|x-z|&\ge |[1-a(x)+u_\eps(x)+u_\eps(z)][u_\eps(x)-u_\eps(z)] + [a(z)-a(x)]u_\eps(x)|,\\     
&\ge |[1-a(x)+u_\eps(x)+u_\eps(z)]||u_\eps(x)-u_\eps(z)| -  |x-z|\frac{|a(z)-a(x)|}{|x-z|}M.
\end{align*}
From the last inequality, it follows that $u_\eps$ is uniformly Lipschitz in $Q:=\{y\in\R^N\,|\, a(y)<1\}$  $u_\eps$  with a Lipschitz constant independent of $\eps$. Thus, $(u_\eps)_{\eps \in [\frac{1}{2}\eps^*,\eps^*)}$ is uniformly bounded in $C^{0,\frac{1}{2}}_{loc}(Q)$. If $Q^c=\emptyset$, then 
 $(u_\eps)_{\eps \in [\frac{1}{2}\eps^*,\eps^*)}$ is uniformly bounded in $C^{0,\frac{1}{2}}_{loc}(\R^N)$. Otherwise, $Q^c \neq \emptyset$  and on $Q^c$ we have $a(x)\equiv 1$. Therefore, on $Q^c$,  $u^2_\eps(x)=J_\eps\star u_\eps$ and the $C^{0,\frac{1}{2}}(\stackrel{\circ}{Q^c})$ norm of $u_\eps$ is bounded independently of $\eps$.   
Hence, 
\begin{equation}\label{bcv-eq-m=0-eps*}
(u_\eps)_{\eps \in [\frac{1}{2}\eps^*,\eps^*)} \quad \text{ is uniformly bounded in } \quad C^{0,\frac{1}{2}}_{loc}(Q)\cap C^{0,\frac{1}{2}}(\stackrel{\circ}{Q^c}).
  \end{equation}
In both case, since $a(x)<0$ for $|x|>>1$, $Q^c$ is a compact set and $|\bar Q \cap Q^c |=0$. From \eqref{bcv-eq-m=0-eps*},  for all sequence $\eps_n\to \eps^*$  by a diagonal extraction procedure  there exists a subsequence still denoted $(u_{\eps_n})_{n\in\N}$ that converges locally uniformly  in $\R^N\setminus( \bar Q \cap Q^c)$ to some non-negative function $v$.  By passing to the limit in \eqref{bcv-eq-asb}  we can see that $v$ is a bounded non negative solution of 
 $$J_{\eps^*}\star v(x) -v(x) +v(x)(a(x)-v(x))=0\quad \text{  in  }\quad \R^N \setminus(\bar Q \cap Q^c).$$
Since $\bar Q \cap Q^c$ is of zero measure $v$ is a solution to
$$ \oplb{v}{\eps^*,\R^N\setminus (\bar Q \cap Q^c)}+v(x)(a(x)-v(x))=0\quad \text{  in  }\quad \R^N \setminus(\bar Q \cap Q^c). $$ 
Since $0\le \lambda_p(\M_{\eps^*} +a(x))\le \lambda_p(\lb{\eps^*, \R^N\setminus (\bar Q \cap Q^c)} +a(x)) $,  we deduce that  $v \equiv 0$  which concludes the proof of the limit.

\begin{remark}
When $a(x)$ is a radially symmetric non-increasing function  we remark that  $\eps^*$ is  a sharp threshold. That is  for all $\eps\ge \eps^*$ then  \eqref{bcv-eq-asb} does not have any positive solutions. Indeed in this situation, the function $a_\eps(x)$ is monotone non increasing with respect to $\eps$. Thus, by (i) of Proposition \ref{bcv-prop-pev}, for all $\eps \ge \eps^*$ we have 
 $$0=\lambda_p(\M+a_{\eps^*}(x))\le\lambda_p(\M+a_{\eps}(x)).$$    
Hence, by Theorem \ref{bcv-thm1}, $0$ is the unique non negative solution to \eqref{bcv-eq-asb} for $\eps\ge \eps^*$. 
\end{remark}

\subsection{The case $0<m<2$}
In this situation, from Theorem \ref{bcv-thm6} we know that 
\begin{align}
&\lim_{\eps\to 0}\lambda_p(\M_{\eps,m} +a(x))= -\sup_{\R^N}a(x)\label{bcv-eq-asl-0m2-1}\\
&\lim_{\eps\to +\infty}\lambda_p(\M_\eps +a(x))= -\sup_{\R^N}a(x)\label{bcv-eq-asl-0m2-2}
\end{align}

As a consequence, for $\eps$ small enough and for large $\eps$  we have  $\lambda_p(\M_\eps +a(x))\le -\frac{\sup_{\R^N}a(x)}{2}<0$. Therefore, by Theorem \eqref{bcv-thm1} there exists a solution of \eqref{bcv-eq-asb} for both  small and large $\eps$.

The limit of $u_\eps$ when  $\eps\to \infty$ is an obvious consequence of $(iv)$ of Lemma \ref{bcv-lem-esti1} and Lemma \ref{bcv-lem-esti2} since  for $\eps$ large enough 
$$(a(x)-\frac{1}{\eps^m})^+\le u_\eps\le a^+(x)+\frac{1}{\eps^{\frac{N}{4}}}. $$

To obtain the limits in $L^2$, we just observe that since by Lemma $u_\eps$ is uniformly bounded in $L^2$ and converges pointwise to $a^+$, we get $u_\eps \rightharpoonup a^+$ in $L^2$. Moreover by Fatou's Lemma, we infer that
$$\int_{\R^N}(a^+)^2(x)\,dx \le \liminf_{\eps\to \infty}\int_{\R^N}u_\eps^2(x)\,dx.$$  
On the other hand, by integrating \eqref{bcv-eq-asb} over $\R^N$ we get for all $\eps$
$$\int_{\R^N}u^2_\eps(x)\,dx =\int_{\R^N}a(x)u_\eps(x)\,dx\le \int_{\R^N}a^+(x)u_\eps(x)\,dx.$$
By the Cauchy-Schwartz inequality, for all $\eps$
 $$\left(\int_{\R^N}u^2_\eps(x)\,dx\right)^{1/2} \le \left(\int_{\R^N}(a^+)^2(x)\,dx\right)^{1/2}$$
 and we get $$\int_{\R^N}(a^+)^2(x)\,dx \le \liminf_{\eps\to \infty}\int_{\R^N}u_\eps^2(x)\,dx\le \limsup_{\eps \to +\infty}\int_{\R^N}u^2_\eps(x)\,dx \le \int_{\R^N}(a^+)^2(x)\,dx.$$
 Hence, $\nlto{u_\eps} \to \nlto{a^+}$ and by the parallelogram identity $u_\eps \to a^+$ in $L^2(\R^N)$ since $u_\eps$ converges weakly to $a^+$ in $L^2$.

 As already mentioned  in Remark \ref{bcv-rem-proof-conv}, it can be seen that $u_\eps$ has a limit when $\eps \to 0$  as soon as  $a$ is smooth enough. Thus,
$$\lim_{\eps \to 0}u_\eps =v(x)$$
   where $v$ is a non-negative bounded solution of \eqref{bcv-eq-eqlim}.

\subsection{The case $m=2$}
In this situation, from Theorem \ref{bcv-thm6} we have  
\begin{align}
&\lim_{\eps\to 0}\lambda_p(\M_{\eps,m} +a(x))= \lambda_1\left(\frac{D_2(J)}{2N}\Delta +a(x)\right)\label{bcv-eq-asl-m2-1}\\
&\lim_{\eps\to +\infty}\lambda_p(\M_\eps +a(x))= -\sup_{\R^N}a(x)\label{bcv-eq-asl-m2-2}
\end{align}

As a consequence, for large $\eps$  we have  $\lambda_p(\M_\eps +a(x))\le -\frac{\sup_{\R^N}a(x)}{2}<0$ and  by Theorem \eqref{bcv-thm1} there exists a solution of \eqref{bcv-eq-asb} for  large $\eps$. For $\eps$ small, the existence of a positive solution is conditioned by the sign of $\lambda_1\left(\frac{D_2(J)}{2N}\Delta +a(x)\right)$. When  $\lambda_1\left(\frac{D_2(J)}{2N}\Delta +a(x)\right)>0$, for $\eps$ small there exists no positive solution of \eqref{bcv-eq-asb}.
The limit of $u_\eps$ when $\eps \to +\infty$ is obtained in the same way as in the case $2>m>0$. Hence, we only focus here on the limit when $\eps\to 0$.  

Assume for the moment that $\lambda_1\left(\frac{D_2(J)}{2N}\Delta +a(x)\right)<0$. We now show that $u_\eps \to v$ where $v$ is the positive solution of 
$$\frac{D_2(J)}{2N}\Delta v +v(a(x)-v)=0 \quad \text{ in } \quad \R^N.$$

Let $(\eps_n)_{n\in \N}$ be a sequence of positive reals converging to $0$. We write $u_n$ instead of $u_{\eps_n}$.
 By Lemma \ref{bcv-lem-esti1}, $\nlto{u_n}$ is bounded uniformly and a simple algebraic computation yields : 
$$\iint_{\R^N\times\R^N}\rho_{n}(z)\frac{(u_n(x+z)-u_n(x))^2}{|z|^2}\,dxdz<C, $$ 
with $C$ independent of $n$. Therefore, for any $R>0$, we see that 

$$\iint_{B_R\times B_R}\rho_{\eps}(z)\frac{(u_n(x+z)-u_n(x))^2}{|z|^2}\,dxdz<C. $$ 
For $R>0$ fixed, since $\nlto{u_n}$ is uniformly bounded in $L^2$, there exists a subsequence $u_n \rightharpoonup v$ in  $L^{2}(B_R)$  and
from the characterisation of Sobolev Space \cite{Ponce2004, Ponce2004a}, we have $u_n\to v$ in $L^2(B_R)$.
 
By a standard diagonal extraction argument, from the sequence $(u_n)_{n\in \N}$ we can then extract a subsequence still denoted $(u_n)_{n\in \N}$ which converges to some $v$ in $L^{2}_{loc}(\R^N)$.  
Moreover, by Lemma \ref{bcv-lem-esti1}, $u_n$ is uniformly bounded and  there exists $\delta(\lambda_1)>0$ independent of $\eps$  such that $ \max_{supp(a^+)}(u_n)>\delta$.

Multiplying \eqref{bcv-eq-asb} by $\varphi \in C^{\infty}_c(\R^N)$ and integrating  yields :  

$$\frac{D_2(J)}{2}\iint_{\R^N\times\R^N}\frac{\rho_n(z)}{|z|^2}u_n(x)[\varphi(x+z)-2\varphi(x)+\varphi(x-z)]\,dxdz+\int_{\R^N}\varphi(x)u_n(x)(a(x)-u_n(x))\,dx=0, $$ 
 where we use \eqref{bcv-eq-asb-id} to compute $\int_{\R^N}\opmem{u_n}{2}(x)\varphi(x)\,dx$.
This leads us to

\begin{multline*}
\frac{D_2(J)}{2}\iint_{\R^N\times\R^N}\frac{\rho_n(z)}{|z|^2}u_n(x)\transposee{z}\nabla^{2}\varphi(x)z\,dxdz+\int_{\R^N}\varphi(x)u_n(x)(a(x)-u_n(x))\,dx
\\= -\frac{D_2(J)}{2}\iint_{\R^N\times\R^N}\frac{\rho_n(z)}{|z|^2}u_n(x)[\varphi(x+z)-2\varphi(x)+\varphi(x-z)-\transposee{z}\nabla^{2}\varphi(x)z]\,dxdz,  
\end{multline*}
where $\nabla^{2}\varphi(x):=(\partial_{ij}\varphi(x))_{i,j}$.
Since $\rho_n(z)$ is radially symmetric, we can see that 
$$ \frac{D_2(J)}{2}\iint_{\R^N\times\R^N}\frac{\rho_n(z)}{|z|^2}u_n(x)\transposee{z}\nabla^{2}\varphi(x)z\,dxdz=\frac{D_2(J)K_{2,N}}{2}\int_{\R^N}u_n(x)\Delta\varphi(x)\,dx$$
with $$K_{2,N}:=\fint_{S^{N-1}}(\sigma\cdot e_1)^2 d\sigma=\frac{1}{N}.$$ 
 Thus,  we get
 
\begin{multline}\label{bcv-eq-lim-m2-eps-0-1}
\frac{D_2(J)}{2N}\int_{\R^N}u_n(x)\Delta\varphi(x)\,dx+\int_{\R^N}\varphi(x)u_n(x)(a(x)-u_n(x))\,dx\\=-\frac{D_2(J)}{2}\iint_{\R^N}\frac{\rho_n(z)}{|z|^2}u_n(x)[\varphi(x+z)-2\varphi(x)+\varphi(x-z)-\transposee{z}\nabla^{2}\varphi(x)z]\,dxdz,  
\end{multline}
Note that since $u_n$ converges to $v$ in $L^2_{loc}(\R^N)$ we have 

 \begin{equation}\label{bcv-eq-lim-m2-eps-0-2}
\int_{\R^N}\varphi(x)u_n(x)(a(x)-u_n(x))\,dx \to  \int_{\R^N}\varphi(x)v(x)(a(x)-v(x))\,dx
\end{equation}

Recall  that $\varphi \in C^{\infty}_c(\R^N)$, so there exists $C(\varphi)$ and $R(\varphi)$ so that
$$|\varphi(x+z)-2\varphi(x)+\varphi(x-z)- \transposee{z}\nabla^{2}\varphi(x)z|<C(\varphi)|z|^{3}\chi_{B_{R(\varphi)}}(x). $$ 

Because $u_n$ is bounded uniformly we obtain
\begin{equation}\label{bcv-eq-lim-m2-eps-0-3}
 \frac{D_2(J)}{2}\iint_{\R^N}\frac{\rho_\eps(z)}{|z|^2}u_n(x)[\varphi(x+z)-2\varphi(x)+\varphi(x-z)-\transposee{z}\nabla^{2}\varphi(x)z]\,dxdz\le CC(\varphi)\int_{\R^N}\rho_n(z)|z| \to 0.
 \end{equation}

Passing to the limit $\eps \to 0$ in \eqref{bcv-eq-lim-m2-eps-0-1}, using  \eqref{bcv-eq-lim-m2-eps-0-2} and \eqref{bcv-eq-lim-m2-eps-0-3}, we get

\begin{equation}\label{bcv-eq-lim-m2-eps-0-4}
\frac{D_2(J)}{2N}\int_{\R^N}v(x)\Delta\varphi(x)\,dx+\int_{\R^N}\varphi(x)v(x)(a(x)-v(x))\,dx=0.
\end{equation}

\eqref{bcv-eq-lim-m2-eps-0-4} being true for any $\varphi \in C^{\infty}_c$ this implies that $v$ satisfies 

$$\frac{D_2(J)}{2N}\Delta v +v(a(x)-v)=0 \quad \text{ a.e. in } \quad \R^N.$$

Since $v$ is bounded, by elliptic regularity $v$ is smooth. 
To conclude we need to prove that $v$ is non zero. To this end,  we claim that
\begin{lemma}
There exists $R_0,\tau$ and $\eps_0$ positive constants so that for all $\eps\le \eps_0$
we have $u_\eps\ge \tau$ almost everywhere in $B_{R_0}(0)$.   
\end{lemma}

From the above claim, we deduce that $v\ge \tau>0 \; a.e.$ and therefore $v\equiv u$, the  unique smooth non-trivial solution of 
$$\frac{D_2(J)}{2N}\Delta u +u(a(x)-u)=0 \quad \text{ in } \quad \R^N.$$
The sequence $(\eps_n)_{n}$ being arbitrary, it follows that $u_\eps \to u$ in $L^2_{loc}( \R^N)$. 

Similarly, if we assume now that $\lambda_1(\frac{D_2(J)}{2N}\Delta  +a(x))=0$ and there exists a sequence $(\eps_n)_{n\in \N}, \eps_n \to 0$ of non trivial solution of \eqref{bcv-eq-asb}. The above argument is valid and we get 
$u_n\to v$ in $L^2_{loc}(\R^N)$ with $v$ a smooth solution to 
$$\frac{D_2(J)}{2N}\Delta v +v(a(x)-v)=0 \quad \text{ in } \quad \R^N.$$
Since $\lambda_1(\frac{D_2(J)}{2N}\Delta  +a(x))=0$, $v\equiv 0$ is the only solution and we get 
 $u_n\to 0$ in $L^2_{loc}(\R^N)$.

Let us complete our proof and establish the Lemma.
\begin{proof}
Let us denote, $\l$ the operator $$ \l[\varphi]:=\frac{1}{\eps^2}\left[\int_{B_R(0)}J_\eps(x-y)\varphi(y)\,dy -\varphi(x)\right].$$
Since  $\sup_{\R^N} a(x)$ is achieved on $\R^N$  we  regularise $a$ by $a_\sigma$ independently of $\eps$, so that for all $\eps$ and $R\ge R_1$ the principal eigenvalue $\lambda_p(\lb +a_\sigma(x))$ is associated with a continuous principal eigenfunction $\varphi_{p,\eps}$ and 
$$|\lambda_p(\lb{R,\eps,2}+a_\sigma(x))-\lambda_p(\lb{R,\eps,2}+a(x))|\le \|a_\sigma(x)-a(x)\|_{\infty}\le \kappa\sigma,$$
 where $\kappa$ is the Lipschitz constant of $a$.

By the Lipschitz continuity of $\lambda_1\left(\frac{D_2(J)}{2N}\Delta + a(x) \right)$ with respect to $a$, we can  choose $\sigma$ small enough so that $$\lambda_1\left(\frac{D_2(J)}{2N}\Delta + a_\sigma(x) \right)\le \frac{1}{2} \lambda_1\left(\frac{D_2(J)}{2N}\Delta + a(x) \right)<0.$$

Recall that 
$$ \lim_{R\to \infty}\lambda_1\left(\frac{D_2(J)}{2N}\Delta + a_\sigma(x), B_R \right)= \lambda_1\left(\frac{D_2(J)}{2N}\Delta + a_\sigma(x) \right),$$
So we can choose  $R_0$ large  so that $$ \lambda_1\left(\frac{D_2(J)}{2N}\Delta + a_\sigma(x), B_{R_0} \right)\le \frac{1}{4} \lambda_1\left(\frac{D_2(J)}{2N}\Delta + a(x) \right).$$ 

Thanks to Theorem \ref{bcv-thm7}, we have $\lim_{\eps \to 0}\lambda_p(\lb{R,\eps,2}+a_\sigma(x))=\lambda_1\left(\frac{D_2(J)}{2N}\Delta + a_\sigma(x), B_R \right)$ so for $\eps$ small,say $\eps \le \eps_0$ by choosing $\sigma$ smaller if necessary, we achieve
$$\lambda_p(\lb{R,\eps,2}+a_\sigma(x))\le  \frac{1}{8} \lambda_1\left(\frac{D_2(J)}{2N}\Delta + a(x) \right) \quad \text{ for all }\quad \eps \le \eps_0.$$ 

Let $\varphi_{p,\eps}$ be the principal eigenfunction associated with $\lb{R_0,\eps,2}+a_\sigma(x)$, then we have 
\begin{equation*}
\lb{R_0,\eps,2}[\varphi_{p,\eps}](x)+a(x)\varphi_{p,\eps}(x)\ge \left[-\frac{1}{8} \lambda_1\left(\frac{D_2(J)}{2N}\Delta + a(x) \right) -\kappa\sigma \right]\varphi_{p,\eps}(x)\quad \text{ for all }\quad \eps \le \eps_0.
 \end{equation*}  
 By choosing $\sigma$ smaller if necessary, $$\left[-\frac{1}{8} \lambda_1\left(\frac{D_2(J)}{2N}\Delta + a(x) \right) -\kappa\sigma \right]\ge -\frac{1}{16} \lambda_1\left(\frac{D_2(J)}{2N}\Delta + a(x) \right) $$
 and we achieve
 
 \begin{equation}\label{bcv-eq-clai-no-triv1}
\lb{R_0,\eps,2}[\varphi_{p,\eps}](x)+a(x)\varphi_{p,\eps}(x)\ge -\frac{1}{16} \lambda_1\left(\frac{D_2(J)}{2N}\Delta + a(x) \right)\varphi_{p,\eps}(x)\quad \text{ for all }\quad \eps \le \eps_0.
 \end{equation}

To conclude our proof, it is then  enough to show that  for some well chosen normalisation of $\varphi_{p,\eps}$ we have
\begin{equation}
\varphi_{p,\eps}(x)\to \varphi_1(x), \quad \text{ a.e.  in }\quad B_{R_0} \label{bcv-eq-clai-no-triv2}
\end{equation}
$\varphi_1$ is a positive principal eigenfunction associated with  $\lambda_1\left(\frac{D_2(J)}{2N}\Delta + a_\sigma(x), B_{R_0} \right)$. 
Indeed, assume for the moment that \eqref{bcv-eq-clai-no-triv2} holds true. Then there exists $\alpha>0$ so that 
$$\alpha\varphi_{p,\eps}(x)\to \alpha\varphi_1(x)<\frac{1}{2} \quad \text{ a.e.  in } \quad B_{R_0}.$$

Now thanks to \eqref{bcv-eq-clai-no-triv1}, we can now adapt the proof  the proof of (iii) of Lemma \ref{bcv-lem-esti1} to get for $\eps$ small, says $\eps \le \eps_1$,

\begin{align}
u_\eps(x) \ge -\frac{\alpha}{32} \lambda_1\left(\frac{D_2(J)}{2N}\Delta + a(x) \right)\varphi_{p,\eps}(x) \quad \text{ a.e.  in }\quad B_{R_0}, \label{bcv-eq-clai-no-triv3}
\end{align} 
which combined with \eqref{bcv-eq-clai-no-triv2} enforces 
$$u_\eps(x) \ge \gamma \varphi_{1}(x) \quad \text{ a.e.  in }\quad B_{R_0},\quad \text{ for all }\quad \eps \le \eps_2,$$
for some $\gamma,\eps_2>0$. Since $\varphi_1>0$ in $B_{R_0}$, the claim holds true in any smaller ball $B_R$. 

To prove \eqref{bcv-eq-clai-no-triv2}, let us normalise $\varphi_{p,\eps}$ by $\nlp{\varphi_{p,\eps}}{2}{B_{R_0}}=1$. Let $k_\eps$ be the function defined by
$$k_\eps(x):=\frac{1}{\eps^2}\int_{\R^N\setminus B_{R_0}}J_\eps(x-y)\,dy.$$
 Multiplying by $\varphi_{p,\eps}$ the equation satisfied by $\varphi_{p,\eps}$ and integrating  over $B_{R_0}$ yields  
\begin{align*}
\frac{D_2(J)}{2}\iint_{B_{R_0}\times B_{R_0}}\rho_\eps(x-y)\frac{|\varphi_{p,\eps}(y)-\varphi_{p,\eps}(x)|^2}{|x-y|^2}\,dxdy&=\int_{B_{R_0}}(a_\sigma(x)+\lambda_{p,\eps})\varphi_{p,\eps}^2(x)\,dx -\int_{B_{R_0}}k_{\eps}(x)\varphi_{p,\eps}^2(x)\,dx\\
 &\le C.
 \end{align*} 
Therefore by the characterisation of Sobolev space \cite{Ponce2004,Ponce2004a}, along a sequence we have $\varphi_{p,\eps}\to \psi$ in $L^2(B_{R_0})$ with $\nlp{\psi}{2}{B_{R_0}}=1$. Moreover by extending $\varphi_{p,\eps}$ and $\varphi$ by $0$ outside $B_{R_0}$ and by arguing as above for any $\varphi \in C^2_c(B_{R_0})$ we have

\begin{equation*}
\begin{split}\frac{D_2(J)}{2}\iint_{B_{R_0}\times \R^N}\frac{\rho_\eps(z)}{|z|^2}\varphi_{p,\eps}(x)[\varphi(x+z)-2\varphi(x)+\varphi(x-z)]\,dxdz = -\int_{B_{R_0}}\varphi(x)\varphi_{p,\eps}(a(x)+\lambda_{p,\eps})\,dx \\+\int_{B_{R_0}}k_\eps(x)\varphi_{p,\eps}(x)\varphi(x)\,dx. \end{split}
\end{equation*} 

Since $\varphi \in C_c^2(B_{R_0})$ we get  for $\eps$ small enough  $supp(k_\eps)\cap supp(\varphi)=\emptyset$. Thus passing to the limit  along a sequence in the above equation yields
 \begin{equation}\label{bcv-eq-clai-no-triv4}
\frac{D_2(J)K_{2,N}}{2}\int_{B_{R_0}}\psi(x)\Delta\varphi(x)\,dx +\int_{B_{R_0}}\varphi(x)\psi(x)(a(x)+\lambda_1)\,dx =0.
\end{equation}

The relation \eqref{bcv-eq-clai-no-triv4} being true for any $\varphi$, it follows that  $\psi$ is the smooth positive eigenfunction associated to $\lambda_1$ normalised by $\nlp{\psi}{2}{B_{R_0}}=1$. $\psi$ being uniquely defined, we get   $\varphi_{p,\eps}\to \psi$ in $L^2(B_{R_0})$ when $\eps \to 0$.
Thus along any sequence $\varphi_{p,\eps}(x)\to \varphi_{1}(x) $ almost everywhere in $B_{R_0}$.

\end{proof}

 \section{Extension to non-compactly supported kernels} 
 \label{bcv-section-ext}
 In this section, we discuss the extension of our persistence criteria to more general dispersal kernel $J$ and prove Theorem \ref{bcv-thm2}. Observe that the construction of positive solution only required that $\lambda_p(\lb{R}+\beta(x))<0$ for some $R$, regardless of what the dispersal kernel $J$ is. Therefore as soon as $ \lim_{R\to \infty} \lambda_p(\lb{R}+\beta(x))<0$ there exists a positive solution to \eqref{bcv-eq} with no restriction on the decay of the kernel.  Similarly, when $ \lambda_p(\M+\beta(x))>0$ the proof of the  non-existence of positive bounded solution essentially  relies on the inequality between $\lambda_p(\M+\beta(x))$ and $\lambda_p'(\M+\beta(x))$ which holds for quite general kernels including those satisfying the assumption $H5$ as proved in \cite{Berestycki2014}.
 Concerning the proof of the uniqueness of the positive solution, it relies on the construction of an integrable uniform super-solution of \eqref{bcv-eq} which guarantes the existence of a positive $L^1$ solution to \eqref{bcv-eq}. 
 Such super-solution still exists for kernels $J$ that satisfies the decay assumption $H5$. Indeed, we can show
 
 \begin{lemma} Assume that $J$ satisfies $H5$ and there exists a periodic function $\mu(x):\R^N\to\R$ such that 
$$\limsup_{|x|\to\infty}(\beta(x)-\mu(x))\leq 0\quad\quad\textrm{and}\quad\quad \lambda_p(\M +\mu(x))>0.$$ 
Then there exists $\bar u \in C_0(\R^N)\cap L^1(\R^N)$, $\bar u>0$ so that  $\bar u$ is a super-solution to \eqref{bcv-eq}. 
\end{lemma}

Observe that the construction of the super-solution thus covers  a  class of nonlinearities $f(x,u)$ larger than those that  satisfy $H4$.  As an  immediate consequence, the persistence criteria obtained in Theorem \ref{bcv-thm1}  still holds for any nonlinearity that satisfies:\\

\hbox{\textbf{(H7)}$\qquad $ There exists $\mu(x)\in C_p(\R^N)\;$ such that : $\qquad$
$ 
\begin{cases}
& \lambda_p(\M +\mu(x))>0,\\
&\limsup_{|x|\to\infty}\left(\frac{f(x,s)}{s}-\mu(x)\right)\leq 0 \quad \text{ uniformly in }s. 
\end{cases}
$
}
\medskip

 From an ecological point of view, such a nonlinearity allows one to consider a more complex niche structure for the species. Thus, we can consider ecological niches that are the superposition of a compact niche structure with  a periodic structure. Assume that in the (unbounded) periodic structure, there is extinction. Then, this framework allows us to discuss perturbation with compact support from the periodic structure and derive conditions for persistence.  The perspective offered by this approach are quite promising  and we believe that it may also  be applied to investigate  a climate change version of \eqref{bcv-eq}.

 \begin{proof}
 The construction of the super-solution in this situation follows the same general scheme as for a compactly supported kernel.
 By assumption since  $\limsup_{|x|\to\infty}(\beta(x)-\mu(x))\leq 0$, for any $\delta>0$ there exists $R_\delta>1$  such that
$$\beta(x)\leq \mu(x)+\delta\quad\quad\text{ for all $x$,}\quad |x|\geq R_\delta.$$

Fix $\delta<\lambda_p(\M+\mu(x))$ and  observe that by the definition of $\lambda_p(\M +\mu(x))$ there exists a constant $\delta<\lambda<\lambda_p(\M+\mu(x))$ and a positive periodic function $\varphi$ such that  
\begin{equation}\label{bcv-eq-phi-lambda}
\opm{\varphi}(x)+(\mu(x)+\lambda)\varphi(x)\leq 0\quad \text{ for all }\quad x\in\R^N.
\end{equation}

Let $w=C\frac{\varphi(x)}{1+\tau|x|^{N+1}}$ with $C,\tau$ to be chosen. 
\begin{align*} 
\opm{w}+(\mu(x)+\delta) w(x) &=C(1+\tau|x|^{N+1})^{-1}\left(\int_{\R^N}J(x-y) \frac{(1+\tau|x|^{N+1})}{(1+\tau|y|^{N+1})}\varphi(y)\,dy-\varphi(x)+(\mu(x)+\delta) \varphi(x) \right), \\
&\le C(1+\tau|x|^{N+1})^{-1}\left(\int_{\R^N}J(z)\left[\frac{(1+\tau|x|^{N+1})}{(1+\tau|x+z|^{N+1})}-1\right]\varphi(x+z)\,dz +(\delta-\lambda)\varphi(x)\right),\\
&\le C(1+\tau|x|^{N+1})^{-1}\left(\tau\int_{\R^N}J(z)\left[\frac{|x|^{N+1}-|x+z|^{N+1}}{(1+\tau|x+z|^{N+1})}\right]\varphi(x+z)\,dz +(\delta-\lambda)\varphi(x)\right),\\
&\le w(x)\left(\tau\int_{\R^N}J(z)\left[\frac{|x|^{N+1}-|x+z|^{N+1}}{(1+\tau|x+z|^{N+1})}\right]\frac{\varphi(x+z)}{\varphi(x)}\,dz +\delta-\lambda\right),
\end{align*}
where we use  \eqref{bcv-eq-phi-lambda} and $\inf_{\R^N}\varphi>0$.

Set $$h(\tau,x):=\tau\int_{\R^N}J(z) \left[\frac{|x|^{N+1}-|x+z|^{N+1}}{(1+\tau|x+z|^{N+1})}\right]\frac{\varphi(x+z)}{\varphi(x)}\,dz +\delta-\lambda.$$

Since $\varphi\in L^\infty(\R^N)$ and $\inf_{\R^N}\varphi>0$,  there exists a positive constant $C_0$ such that   
$$\frac{\varphi(x+z)}{\varphi(x)}\le C_0\quad \text{ for all }\quad x,z\in\R^N.$$
For all $x \in \R^N$, we have 
\begin{equation}
h(\tau,x)\le C_0\tau\int_{\R^N}J(z)\left[\frac{|x|^{N+1}-|x+z|^{N+1}}{(1+\tau|x+z|^{N+1})}\right]\,dz +\delta-\lambda.\label{bcv-eq-gen-0}
\end{equation}

Let  $$I:=C_0\tau\int_{\R^N}J(z)\left[\frac{|x|^{N+1}-|x+z|^{N+1}}{(1+\tau|x+z|^{N+1})}\right]\,dz,$$  then we have 

$$I=C_0\tau\int_{\{|x|\le 2|z|\}}J(z)\left[\frac{|x|^{N+1}-|x+z|^{N+1}}{(1+\tau|x+z|^{N+1})}\right]\,dz +C_0\tau\int_{\{|x|> 2|z|\}}J(z)\left[\frac{|x|^{N+1}-|x+z|^{N+1}}{(1+\tau|x+z|^{N+1})}\right]\,dz.$$

 Let us estimate the first  integral. Since $|x|\le 2|z|$ we have
   
\begin{equation}
C_0\tau\int_{\{|x|\le 2|z|\}}J(z)\left[\frac{|x|^{N+1}-|x+z|^{N+1}}{(1+\tau|x+z|^{N+1})}\right]\,dz\le C_0\tau 2^{N+1}\int_{\R^N}J(z)|z|^{N+1}\,dz. \label{bcv-eq-gen-1}
\end{equation}

Let us now estimate the second term. Since  $|x+z|^{N+1}\ge (|x|-|z|)^{N+1}$, we have 
\begin{align*}
C_0\tau\int_{\{|x|> 2|z|\}}J(z)\left[\frac{|x|^{N+1}-|x+z|^{N+1}}{(1+\tau|x+z|^{N+1})}\right]\,dz&\le C_0 \sum_{i=1}^{N+1}\binom{N+1}{i} \int_{\{|x|> 2|z|\}}J(z)(-1)^{i+1}|z|^{i}\left[\frac{\tau|x|^{N+1-i}}{(1+\tau|x+z|^{N+1})}\right]\,dz, \\
&\le C_0 \sum_{i=1}^{N+1}\binom{N+1}{i} \int_{\{|x|>2|z|\}}J(z)|z|^{i}\left[\frac{\tau|x|^{N+1-i}}{(1+\tau|x+z|^{N+1})}\right]\,dz.
\end{align*}

Since $|x|>2|z|$, we have $$\frac{1}{1+\tau|x+z|^{N+1}}\le \frac{2^{N+1}}{2^{N+1}+\tau|x|^{N+1}}$$ and for $|x|\ge R_0>1$ 
\begin{align*}
C_0\tau\int_{\{|x|> 2|z|\}}J(z)\left[\frac{|x|^{N+1}-|x+z|^{N+1}}{(1+\tau|x+z|^{N+1})}\right]\,dz&\le C_0 2^{N+1} \sum_{i=1}^{N+1}\binom{N+1}{i} \int_{\{|x|> 2|z|\}}J(z)\frac{|z|^{i}}{|x|^i}\left[\frac{\tau|x|^{N+1}}{(2^{N+1}+\tau|x|^{N+1})}\right]\,dz\\
&\le \frac{C_0 2^{N+1}}{R_0} \sum_{i=1}^{N+1}\binom{N+1}{i} \int_{\R^N}J(z)|z|^{i}\left[\frac{\tau|x|^{N+1}}{2^{N+1}+\tau|x|^{N+1}}\right]\,dz.
\end{align*}
Since for all $|x|,$  $$\left[\frac{\tau|x|^{N+1}}{2^{N+1}+\tau|x|^{N+1}}\right]<1,$$
we get for $|x|\ge R_0$
\begin{equation}\label{bcv-eq-gen-2}
C_0\tau\int_{\{|x|> 2|z|\}}J(z)\left[\frac{|x|^{N+1}-|x+z|^{N+1}}{(1+\tau|x+z|^{N+1})}\right]\,dz\le \frac{C_0 2^{N+1}}{R_0} \sum_{i=1}^{N+1}\binom{N+1}{i} \int_{\R^N}J(z)|z|^{i}\,dz.
\end{equation}

Combining \eqref{bcv-eq-gen-1}, \eqref{bcv-eq-gen-2} and \eqref{bcv-eq-gen-0}, we get for $|x|>R_0$

$$h(x,\tau)\le \frac{C_0 2^{N+1}}{R_0} \sum_{i=1}^{N+1}\binom{N+1}{i} \int_{\R^N}J(z)|z|^{i}\,dz+  C_0\tau 2^{N+1}\int_{\R^N}J(z)|z|^{N+1}\,dz+\delta -\lambda.$$
 Thanks to (H5),  for $\tau$ small enough, says $\tau \le \tau_1$ and $R_0$ large enough we derive
 $h(x,\tau)\le  \frac{\delta -\lambda}{2} <0$,

Hence,  for all $\tau \le \tau_1$, we have 
 
 \begin{equation}\label{bcv-eq-supersol-mu}
 \opm{w}+(\mu(x)+\delta) w(x)\le w(x)h(x,\tau)\le w(x)\frac{\delta -\lambda}{2}<0\quad \text{for all } \quad x\in \R^N\setminus B_{R_0}.  
\end{equation}  
 Fix now $\tau\le \tau_1$ and fix $R_0>R_\delta$ so that $h(x,\tau)<0$ in $\R^N\setminus B_{R_0}(0)$. Let $\kappa_0:=\sup_{\R^N\setminus B_{R_0}(0)} \frac{\varphi(x)}{1+\tau|x|^{N+\alpha}}$
Let $0<\kappa<\kappa_0$ and consider the set $$\O_\kappa:=\left\{x\in \R^N\, |\, \frac{\varphi(x)}{1+\tau|x|^{N+\alpha}} \le \kappa\right\}.$$
By construction, since $\varphi>0$ in $\R^N$, we can choose $\kappa$ small so that $$\O_\kappa\subset \R^N\setminus B_{R_0}(0).$$
Moreover, $\R^N\setminus \O_\kappa$ is a bounded domain and  $M:=\sup_{\R^N\setminus \O_\kappa}S(x)$ is well defined. Choose now  $C$ such that $C=\frac{2M}{\kappa}$ and consider the continuous function 
$$\bar u (x):= \begin{cases} C\frac{\varphi(x)}{1+\tau|x|^{N+\alpha}} \quad \text{ in } \quad \O_\kappa,\\
C\kappa \quad \text{ in } \quad \R^N\setminus\O_\kappa.
 \end{cases}
 $$

We can check that $\bar u$ is a super-solution to \eqref{bcv-eq}. Indeed, for any $x \in  \R^N\setminus \O_\kappa$, we have $\bar u =C\kappa=2M> \sup_{\R^N\setminus \O_\kappa} S(x)$  which implies that  $f(x,C\kappa)\le 0$ and  
$$\opm{\bar u}(x) +f(x,\bar u(x))\le \int_{\R^N}J(x-y)\bar u(y)\,dy -C\kappa +f(x,C\kappa)\le f(x,C\kappa)\le 0.$$
On the other hand, for  $x\in \O_\kappa$ we directly have 
\begin{align*}
 \opm{\bar u}(x) +f(x,\bar u(x))\le \opm{\bar u}(x) +\beta(x)w(x) &\le \opm{w}+(\mu(x)+\delta)w(x),\\
 & \le h(x,\tau) w(x)\le 0. 
 \end{align*}
 This completes the proof.
 \end{proof}

\section{Conclusion}

In this paper, we obtain an optimal persistence criteria for a population that has a long range dispersal. The dynamics of the population is described by a reaction dispersion equation with a convolution for dispersal and a Fisher-KPP type nonlinearity  that describes the reproduction and mortality of the population. The model reads :
$$\frac{\partial u}{\partial t}(t,x)=\opm{u} +f(x,u(t,x)),$$
with $\opm{u}=J\star u - u$.

We consider here the case of a bounded ecological niche, that is  when the environment is lethal to the population outside a bounded region.  In our model, this fact  is translated by an assumption on the   Fisher-KPP  nonlinearity, th
This fact is translated in our model  by assuming  that the Fisher -KPP nonlinearity $f$ satisfies:

   $$\limsup_{|x|\to \infty} \frac{f(x,s)}{s}<0,\quad \text{ uniformly in } \quad s\ge 0.$$


When the dispersal kernel is compactly supported,  we prove that the existence of a  positive solution of the above equation is characterised by the sign of the generalised principal eigenvalue $\lambda_p(\m + \partial_sf(x,0))$ defined by 
$$\lambda_p(\m+\partial_sf(x,0)):=\sup \{\lambda \in \R \, | \exists \, \varphi \in C(\R^N), \varphi>0, \text{ such that }\, \opm{\varphi}+\partial_sf(x,0)\varphi+\lambda\varphi \le 0 \}.$$

Moreover, when such a positive stationary solution exists, it is unique. In addition, we describe completely the long time behaviour of positive solution of the above nonlocal Fisher-KPP equation.  We also obtain persistence criteria for fat-tailed kernel, in terms of the sign of $\lim_{R\to\infty}\lambda_p(\lb{R}+\partial_sf(x,0))$. However, due to the lack of Harnack type \textit{ a priori } estimate, the optimality of this  criteria  is still an open problem.  A better understanding of the properties of the generalised principal eigenvalue $\lambda_p$ in such context would allow one to resolve this issue. In particular, the optimality of this criteria would follow from  proving that $$\lim_{R\to\infty}\lambda_p(\lb{R}+\partial_sf(x,0))=\lambda_p(\m+\partial_sf(x,0)).$$   

In the context of compactly supported kernel, we also analyse the effect of the range of dispersal  on the persistence of a species for some rescaled dispersal kernels e. g. $\frac{1}{\eps^m}J_\eps(z)$.  These rescaled kernels arise, when the dispersal is conditioned by a \textit{ dispersal budget} in which the cost functions are of the form $|z|^m$. For $0\le m <2$,  we prove that small spreaders, i.e.   $\eps$ small, always survive. This is not necessarily true when $m=2$. In that case, it may happen that having a small dispersal   leads to extinction. Conversely, when $m>0$, we prove that large spreaders, i.e  $\eps$ large, always survive.             
We also provide the asymptotics of the solution of the associated nonlocal Fisher-KPP equation. These asymptotics provide valuable informations, when we try to compare  the different dispersal strategies and for the search of a Evolutionary Stable Strategy. 

Many new open problems and new directions come up naturally as the continuation of the present study. For instance, to clarify the effect of the dispersal budget on the dispersal strategies, we would need a deeper understanding of the generalised principal eigenvalue. We suspect that for a quadratic cost function, at least in some situation, the effect of the dispersal budget should be the opposite of the one usually observed for unconditional dispersal strategies.  That is, the larger \textit{ spreader } should be always favoured.


\section*{Acknowledgements}

The research leading to these results has received funding from the European Research Council
under the European Union's Seventh Framework Programme (FP/2007-2013) / ERC Grant
Agreement n°321186 : "Reaction-Diffusion Equations, Propagation and Modelling" held by Henri Berestycki.
J. Coville acknowledges support from the “ANR JCJC” project MODEVOL: ANR-13-JS01-0009. 
\section*{}
\bibliographystyle{amsplain}
\bibliography{bcv.bib}

\end{document}